\documentclass[english,10pt]{amsart}
\title[Log Clemens conjecture and log connectivity ]{Infinitesimal invariants of mixed Hodge structures
 II: Log Clemens conjecture and log connectivity}
\author{Rodolfo Aguilar}
\date{}
\address{\parbox{\linewidth}{Centro de Investigacion en Matemáticas, A.C. (CIMAT), Jalisco S/N, Col. Valenciana CP: 36023 
Guanajuato, Gto, México}}

\email{\href{mailto:aaguilar.rodolfo@gmail.com}{aaguilar.rodolfo@gmail.com}}
\urladdr{\url{https://sites.google.com/view/rodolfo-aguilar/}}
      
\usepackage{hyperref}

\usepackage[english]{babel}
\usepackage{amssymb}
\usepackage{amsthm}
\usepackage{mathrsfs}
\usepackage{amsmath}
\usepackage{amsfonts}
\usepackage{graphicx}
\usepackage{subcaption}
\graphicspath{ {C:\Users\xckyu\OneDrive\Documentos\IF} }
\usepackage{tikz-cd}
\tikzcdset{scale cd/.style={every label/.append style={scale=#1},
    cells={nodes={scale=#1}}}}
\usepackage{tikz}
\usetikzlibrary{calc} 
\usetikzlibrary{arrows.meta,bending,decorations.markings,intersections} 

\tikzset{
    arc arrow/.style args={%
    to pos #1 with length #2}{
    decoration={
        markings,
         mark=at position 0 with {\pgfextra{%
         \pgfmathsetmacro{\tmpArrowTime}{#2/(\pgfdecoratedpathlength)}
         \xdef\tmpArrowTime{\tmpArrowTime}}},
        mark=at position {#1-\tmpArrowTime} with {\coordinate(@1);},
        mark=at position {#1-2*\tmpArrowTime/3} with {\coordinate(@2);},
        mark=at position {#1-\tmpArrowTime/3} with {\coordinate(@3);},
        mark=at position {#1} with {\coordinate(@4);
        \draw[-{Stealth[length=#2,bend]}]       
        (@1) .. controls (@2) and (@3) .. (@4);},
        },
     postaction=decorate,
     }
}

\theoremstyle{plain} 

\newtheorem{thm}{Theorem}[section]
\newtheorem{lem}[thm]{Lemma}
\newtheorem{prop}[thm]{Proposition}
\newtheorem{cor}[thm]{Corollary}

\newtheorem{prob}{Problem}[section]

\theoremstyle{definition}

\newtheorem{exmp}{Example}[section]

\theoremstyle{remark}
\newtheorem{rem}[thm]{Remark}

\newtheoremstyle{case}{}{}{}{}{}{:}{ }{}
\theoremstyle{case}

\DeclareMathOperator{\Ima}{Im}

\DeclareMathOperator{\Hom}{Hom}

\DeclareMathOperator{\Ker}{Ker}

\DeclareMathOperator{\Hilb}{Hilb}
\DeclareMathOperator{\Alb}{Alb}
\DeclareMathOperator{\AJ}{AJ}
\DeclareMathOperator{\dAJ}{dAJ}
\DeclareMathOperator{\Ext}{Ext}

\DeclareMathOperator{\Gr}{Gr}
\DeclareMathOperator{\DR}{DR}
\DeclareMathOperator{\prim}{prim}
\DeclareMathOperator{\Pic}{Pic}
\DeclareMathOperator{\RHom}{RHom}

\newcommand{\abs}[1]{\left\vert#1\right\vert}

\newcommand{\QQ}{\mathbb{Q}}
\newcommand{\Oo}{\mathscr{O}}

\begin{document}

\begin{abstract}
Following previous work, we continue the study of infinitesimal methods in mixed Hodge theory. In the first part, inspired by the deformation theory of curves on Calabi-Yau threefolds, we study deformations of smooth $\mathbb{Q}$-log Calabi-Yau pairs $(X,Y)$. We prove unobstructedness results for these pairs under Fano hypotheses. We define families of infinitesimal Abel-Jacobi maps associated with these deformation problems and show that they control the first-order deformations of smooth curves embedded in the pair. Crucially, for the $\frac{1}{2}$-log Calabi-Yau case, we establish an exact duality between deformations and obstructions, recovering the symmetry found in the absolute Calabi-Yau setting. We apply this framework to the cubic threefold, proposing a relative generalization of the Clemens conjecture regarding the injectivity of the infinitesimal Abel-Jacobi map, and establishing a criterion for its non-vanishing.

In the second part, we define infinitesimal invariants for normal functions using extension classes and the log-Leray filtration. Relying on the theory of generalized Jacobian rings developed by Asakura and Saito, we prove a logarithmic Nori connectivity theorem for the universal family of open hypersurfaces, we also deduce a sharp algebraic criterion for the properness of the Hodge loci for open hypersurfaces, generalizing the proof of Carlson-Green-Griffiths-Harris.
\end{abstract}
      
\maketitle

\tableofcontents
\section{Introduction}
\subsection{Motivation} This note continues the line of investigation initiated in \cite{AGG24} and \cite{AGG24b}. In \cite{AGG24}, we introduced infinitesimal invariants for pairs $(X,Y)$, where $X$ is a smooth projective variety and $Y$ is a smooth divisor. Specifically, when $X$ is a smooth Fano threefold and $Y$ is an anticanonical divisor, we constructed a cubic form analogous to the Yukawa-Griffiths cubic for Calabi-Yau threefolds. We used this cubic and proved a generic Torelli theorem for a generic pair $(X,Y)$ with $X$ a cubic threefold. Subsequently, in \cite{AGG24b}, we studied smooth curves $C$ in Fano threefolds $X$ with $Y \in \abs{-K_X}$. There, we connected the Abel-Jacobi map of $X$ to the log-Calabi-Yau geometry of the pair, specifically relating it to the symplectic form on $Y$. These results suggested that the geometry of curves should be governed not just by the Abel-Jacobi map of $X$, but by a generalized Abel-Jacobi map associated with the pair $(X,Y)$.

This perspective is strongly motivated by the classical geometry of curves on Calabi-Yau threefolds. If $X$ is a smooth projective Calabi-Yau threefold and $C \subset X$ is a smooth curve, the deformation theory is controlled by Hodge theory. A central guide in this area is the Clemens conjectures \cite{C87}, particularly the prediction that the infinitesimal Abel-Jacobi map\begin{equation}\label{eq:IntAJ}\dAJ: H^0(N_{C/X}) \to H^1(\Omega_X^2)^*\end{equation}is injective for any smooth quintic threefold. Furthermore, in the Calabi-Yau setting, one has the perfect duality $H^0(N_{C/X}) \cong H^1(N_{C/X})^*$, which implies that the dual of $\dAJ$ essentially computes the obstructions to deforming the curve.

Our goal is to transpose this robust picture to the relative setting. While standard log-Calabi-Yau pairs share many properties with the absolute case, the deformation duality is not perfectly symmetric. This motivates us to consider $\mathbb{Q}$-log Calabi-Yau pairs (where $K_X + mY \simeq \mathcal{O}_X$). As we recently proposed in \cite{A25}, this seems to be the correct framework to generalize the Clemens conjectures. In the present work, we show that for specific values of $m$, we recover the classical duality and can effectively use the log Abel-Jacobi map to control deformations of curves within the pair.

\subsection{Results}
Our results are divided into two main parts. In the first part, we study smooth curves in smooth pairs $(X,Y)$ using infinitesimal methods. Our main object of study is the smooth $\mathbb{Q}$-log CY pair. In the second part, we apply these methods to the universal family of open hypersurfaces.

Let us describe the first part. Let \(X\) be a smooth projective variety of dimension $n$ and \(Y\subset X\) a smooth divisor. We call $(X,Y)$ a smooth $\mathbb{Q}$-log CY pair if
\[
K_X + mY \simeq \Oo_X
\]
for some integer \(m\). Fix an integer \(k\). In Section \ref{ss:FormulasTw}, we define the following bundles and exponents $s_j$:
\[
\begin{aligned}
S_0 &:= T_X \simeq \Omega_X^{\,n-1}(mY), & s_0&:=m,\\
S_1 &:= T_X(-\log Y) \simeq \Omega_X^{\,n-1}(\log Y)((m-1)Y), & s_1&:=m-1,\\
S_2 &:= T_X(-kY) \simeq \Omega_X^{\,n-1}((m-k)Y), & s_2&:=m-k,\\
S_3 &:= T_X(-\log Y)(-kY) \simeq \Omega_X^{\,n-1}(\log Y)((m-k-1)Y), & s_3&:=m-k-1.
\end{aligned}
\]

We first prove an unobstructedness result for these pairs.

\begin{thm}[Theorem \ref{thm:vanishing}]
Suppose \(X\) is Fano. Then:

\medskip\noindent (A) If \(s_j\ge 1\) (i.e. \(\Oo_X(s_jY)\) is ample), then
\[
H^q(X,S_j)=0\qquad\text{for every }q>1.
\]

\medskip\noindent (B) If one of the equalities
\[
s_1=0\ (\text{i.e. } m=1),\qquad s_2=0\ (\text{i.e. } m=k),\qquad s_3=0\ (\text{i.e. } m=k+1)
\]
holds, then for the corresponding \(j\in\{1,2,3\}\)
\[
H^q(X,S_j)=0\qquad\text{for }q=2,\dots,n-1.
\]

In particular, in all cases \(H^2(X,S_j)=0\), which implies that the deformations associated with $S_j$ are unobstructed.
\end{thm}

Next, we relate these deformations to Hodge theory. For any smooth pair $(X,Y)$ and a codimension $q$ smooth subvariety $Z\subset X$ transverse to $Y$, the twists of $\Omega_X^{n-q+1}$ appearing in $S_0,\ldots, S_3$ define infinitesimal log Abel-Jacobi maps. We obtain maps of the form:
\[
H^0(Z,N_{Z/X}(-Y))\to H^{n-q}(X,\Omega_X^{n-q+1}(\log Y))^*.
\]

Restricting to the case where $(X,Y)$ is a smooth $\mathbb{Q}$-log CY $3$-fold pair and $C \subset X$ is a smooth curve, we study the complexes
\begin{equation}\label{int:complexes}
\begin{array}{l}
T_X\to N_{C/X} \\
T_X(-\log Y)\to N_{C/X}
\end{array}
\end{equation}
and their twists by $\Oo_X(-Y)$. These complexes control the deformations of the curve $C$ together with the pair $(X,Y)$. In Section \ref{ss:ObsAJ}, generalizing an argument of Mark Green, we relate the obstruction of these deformations to the Abel-Jacobi map.

\begin{thm} 
The following diagram is commutative:
\[
\begin{tikzcd} 
T_X(-\log Y)(-kY) \ar[r] \ar[d,"\cong"] & N_{C/X}(-kY) \ar[d,"\cong"]\ar[r] & 0 \\
\Omega_X^2(\log Y)(m-k-1)Y \ar[r] & K_C((m-k)Y)\otimes N_{C/X}^*\ar[r] & 0
\end{tikzcd}
\]
\end{thm}

\begin{cor} 
The dual of the infinitesimal log $\AJ$ map is naturally identified with the obstruction map:
\[
\dAJ^*:H^1(\Omega_X^2(\log Y)((m-k-1)Y)) \to H^0(N_{C/X}((k-m)Y))^*
\]
identifies with
\[
H^1(T_X(-\log Y)(-kY))\to H^1(N_{C/X}(-kY)).
\]
\end{cor}

We apply this framework to the cubic threefold. Let $X$ be a smooth cubic threefold, $H\subset X$ a smooth hyperplane section, and $C\subset X$ a smooth rational curve. We state the analogue of the Clemens conjecture in this relative situation: the infinitesimal Abel-Jacobi map
\[
\dAJ:H^0(N_{C/X}(-H))\to H^1(\Omega_X^2(\log H))
\]
is injective. 

The choice of $H \in \abs{-\frac{1}{2} K_X}$ (i.e., $m=2$) is justified by the following duality result, which recovers the symmetry found in absolute Calabi-Yau threefolds.

\begin{thm}[Theorem \ref{thm:duality}] For \emph{any} smooth curve $C\subset X$, we have the following duality:
\[
H^0(N_{C/X}((k-m)Y))\cong H^1(N_{C/X}(-kY))^*
\]
\end{thm}

\begin{cor}\label{cor:IntroDuality} 
For $m=2$ and $k=1$, we have:
\[
H^0(N_{C/X}(-H))\cong H^1(N_{C/X}(-H))^*
\]
\end{cor}
As noted in Remark \ref{rem:dual}, this symmetric duality fails for other values of $m$. Consequently, if the general smooth rational curve $C\subset X$ is maximally balanced (i.e., $N_{C/X}(-H)=\Oo(-1)\oplus \Oo(-1)$), the duality implies an isomorphism between the tangent spaces of the deformations of pairs $(X,H)$ and the triples $(X,H,C)$.

Finally, in Section \ref{S:infAJCri}, we generalize a criterion of Clemens \cite{C89} to detect the non-triviality of $\dAJ$. Suppose $S\subset \mathbb{P}^4$ is a surface such that $S\cap X=C+C'$, where $C'$ is a residual curve intersecting $C$ transversely at points $\{p_1,\ldots,p_m\}$.

\begin{thm} 
The infinitesimal Abel-Jacobi map $\dAJ$ is non-trivial at $C$ if:
\begin{enumerate}
\item There exists $\omega \in H^0(\Omega_{\mathbb{P}^4}^4(2X)(\log Y))$ vanishing at $p_1, \dots, p_m$ but not at a point $p_0 \in C \cap C'$, and
\item There exists $v\in H^0(N_{C/X}(-Y))$ such that $v \wedge T_{p_0}C'$ gives a non-zero element in the geometric fiber of $\det N_{C/X}(-Y)$ at $p_0$.
\end{enumerate}
\end{thm}

We apply this criterion to a line lying in a cone over an elliptic curve inside the Fermat cubic threefold to prove non-vanishing in that case.

Let us now describe the second part of our results. In Section \ref{S:NorFun}, we define infinitesimal invariants for families of pairs. First, following standard constructions in mixed Hodge theory, we define log-normal functions associated with algebraic cycles homologous to zero as extension classes. Then, we construct infinitesimal invariants using the log-Leray filtration, generalizing the construction of Voisin \cite{V02} to the logarithmic setting. 

In Section \ref{S:logConnect}, we turn our attention to the topology of the universal family of open hypersurfaces. Infinitesimal methods in Hodge theory are particularly effective when combined with the Leray spectral sequence. A celebrated example is Nori's connectivity Theorem \cite{N93}, which has had numerous applications in algebraic geometry. We prove an analogue of Nori's theorem in the open setting.

First, we prove a Zariski theorem of Lefschetz type via Hodge theoretic methods. While this is classical let us write it here to fix notation.

\begin{thm}[Zariski thm of Lefschetz type]
Let $X$ be smooth projective of dimension $n+1$. Consider $Y,D$ hypersurfaces with $D$ SNC, $Y$ general and ample.
Consider $X^\circ=X\setminus H, Y^\circ=Y\setminus (H\cap Y)$. Then the map
$$H^i(X^\circ, \mathbb{Q})\to H^i(Y^\circ, \mathbb{Q}) $$
induced by the inclusion $i^\circ:Y^\circ \hookrightarrow X^\circ$
is an isomorphism for $i\leq n-1$ and injective for $i=n$
\end{thm}

The SNC hypothesis can be removed by giving the same proof but using Hodge modules.

We now have to restrict to the universal open family of hypersurfaces of $\mathbb{P}^{n+1}$. Let $X=\mathbb{P}^{n+1}$ and fix a SNC divisor $D=\sum_{i=1}^e D_i \subset X$. Let $L\to X$ be an ample divisor and let $S\subset H^0(X,L)$ be the locus of smooth hyperplane sections. Denote by $\mathcal{X}=X\times S\to S$ the trivial family and consider a global divisor $\mathcal{D}=D\times S$. Define $\mathcal{X}^\circ:=\mathcal{X}\setminus \mathcal{D}$. 

 Consider the universal family of hypersurfaces $\mathcal{Y} \to S$ defined by sections of a line bundle $L$. Let $\mathcal{Y}^\circ = \mathcal{Y} \setminus (\mathcal{Y} \cap (D \times S))$ be the family of open hypersurfaces.

\begin{thm}[Log Connectivity]\label{thm:IntConnectivity}
If the degrees of $D_i$ and $L$ are sufficiently large, the map
$$H^i(\mathcal{X}^\circ, \mathbb{Q})\to H^i(\mathcal{Y}^\circ,\mathbb{Q}) $$
is an isomorphism for $i<2n$ and injective for $i=2n$.
\end{thm}

The proof relies on the theory of generalized Jacobian rings for open complete intersections developed by Asakura and Saito \cite{AS06}. We expect this connectivity theorem to have many new applications.

Finally, in Section \ref{S:RelNL}, we study the Hodge loci in the universal family of open hypersurfaces. 

Generalizing the method of \cite{CGGH83} as explained in Voisin's book \cite{V02}, which relates the codimension of the Hodge loci to the infinitesimal variation of Hodge structure, we prove a criterion for the properness of the Hodge locus in the log setting.

\begin{thm}\label{thm:IntHodgeLoci}
Let $S$ be the parameter space of smooth hypersurfaces $X$ meeting a fixed divisor $Y$ transversely. Fix a flat section $\lambda$ of the primitive cohomology of the complement. Under the numerical conditions where the Asakura-Saito duality for the Jacobian ring is perfect, the Hodge locus
$$S_\lambda^p = \{ s \in S \mid \lambda_s \text{ lies in } F^p H^{n-1}(X_s \setminus Y\cap X_s) \} $$
is a proper analytic subset of $S$.
\end{thm}

The original motivation to have Theorem \ref{thm:IntConnectivity} (respectively \ref{thm:IntHodgeLoci}) was to obtain criteria for the non-vanishing (resp. vanishing) of the log-AJ maps. This will be done elsewhere.

\subsection{Sections}
Let us describe in more detail the content of each section.

\subsubsection{Section \ref{S:infAJ}}
We begin by recalling the definition of the infinitesimal Abel-Jacobi map in the classical projective setting. Then, we adapt this definition to the case of pairs. In fact, a series of infinitesimal log-AJ maps arise. This is due to two facts:
\begin{itemize}
\item When considering a smooth pair $(X,Y)$, two types of cohomology interact via duality: $H^k(X\setminus Y,\mathbb{C})$ and $H^{2n-k}(X,Y,\mathbb{C})$ (cohomology with compact supports).
\item When deforming a smooth curve $C$ inside a pair $(X,Y)$, one can impose various conditions regarding the order of contact with $Y$ along the deformation.
\end{itemize}
These phenomena propagate throughout the paper. For brevity, in subsequent sections, we restrict our attention to the specific cases relevant to our main results.

\subsubsection{Section \ref{S:Dualities}}
This section contains the core technical results regarding the deformation theory of $\mathbb{Q}$-log Calabi-Yau pairs. We prove the duality
$$H^0(N_{C/X}(-H))\cong H^1(N_{C/X}(-H))^*$$
for the case $m=2$. We then generalize the argument of Mark Green to this setting, showing that the dual of the infinitesimal Abel-Jacobi map controls the obstructions to deforming curves. We conclude the section by proving the unobstructedness of $\mathbb{Q}$-log CY pairs under Fano hypotheses.

In an appendix, we follow the ideas of \cite{C83} to construct smooth rational curves on cubic threefolds of arbitrarily high degree.

\subsubsection{Section \ref{S:infAJCri}}
To provide further geometric criteria for the non-vanishing of the Abel-Jacobi map, we generalize a proposition of Clemens \cite{C89} to the relative situation. We utilize the extension class of the normal bundle of $C$ in $\mathbb{P}^4$ relative to $X$. By pairing the restriction of logarithmic forms with this extension class, we obtain an explicit method to compute $\dAJ$.

We formulate a criterion for the non-triviality of $\dAJ$ based on the existence of a residual curve $C'$ intersecting $C$ transversely. We apply this criterion to the case of the Fermat cubic threefold, specifically for a line lying in a cone over an elliptic curve.

\subsubsection{Section \ref{S:NorFun}}
In this section, we treat log-normal functions as extension classes of mixed Hodge structures and define their infinitesimal invariants. We also define invariants using the holomorphic Leray filtration for families of pairs $(\mathcal{X},\mathcal{Y})\to B$, following the construction in Voisin's book \cite{V02}. We compute an example for a smooth log-CY threefold using infinitesimal Torelli.

\subsubsection{Section \ref{S:logConnect}}
Here we combine infinitesimal methods in Hodge theory with the Leray spectral sequence to study the topology of open hypersurfaces. We prove an analogue of Nori's connectivity Theorem \cite{N93} in the open setting. Our main tool is the theory of generalized Jacobian rings for open complete intersections developed in \cite{AS06}, which provides the necessary vanishing results for the graded pieces of the de Rham complex.

\subsubsection{Section \ref{S:RelNL}}
In the final section, we apply the results of \cite{AS06} to the study of the relative Hodge locus. We verify that the algebraic description of the infinitesimal variation of Hodge structure via Jacobian rings satisfies the conditions required to bound the dimension of the Hodge locus. In particular, we show that the Hodge locus is proper whenever the degrees of the irreducible components of the boundary divisor are sufficiently high.

\section*{Acknowledgement} I am indebted to Mark Green and Phillip Griffiths for comments on early drafts, for sharing their ideas, and for their guidance and support throughout this research. I also thank Cristhian Garay and Pedro Luis del Ángel for many stimulating discussions and for the excellent working environment at CIMAT.

This work was supported by a postdoctoral fellowship, Estancias Posdoctorales por México 2025, from the Secretaría de Ciencia, Humanidades, Tecnología e Innovación (SECIHTI), Mexico.
\section{Infinitesimal log Abel-Jacobi}\label{S:infAJ}
\subsection{Projective case}
Let us give a brief recall of how the differential of the Abel-Jacobi map can be defined. 

Let $X$ be a smooth projective variety of dimension $n$ and let $Z\subset X$ be a codimension $q$ cycle. Suppose that we vary $Z$ in a family $Z_\lambda\subset X$ with $\lambda \in D$ the unit disk, and $z_0=z$.

Then the dual of the differential of the Abel-Jacobi map is given by a diagram
\begin{equation}\label{diagram:D}
\begin{tikzcd} & H^{n-q+1,n-q}(X)\cong H^{n-q}(X,\Omega_X^{n-q+1}\ar[dd])\ar[ld] \\
 \Omega_{D,0}^1& \\
& H^{n-q}(Z,\Omega_Z^{n-q}\otimes N_{Z/X}^*)=H^0(Z,N_{Z/X})^* \ar[lu,"\rho"]
\end{tikzcd}
\end{equation}
where $\rho$ is given by the Kodaira-Spencer map.
\subsection{Restriction}
Now we also consider $Y\subset X$ a smooth divisor. 

We have two exact sequences related to the tangent sheaves that will play an important role along all the note. The first one is: 
\begin{equation}\label{eq:II}
0 \to T_Z(-\log Y)\to T_X(-\log Y)|_Z\to N_{Z/X}\to 0.
\end{equation}

Here we have abused the notation and write $T_Z(-\log Y)$ instead of $T_Z(-\log Y\cap Z)$.

The second one:
\begin{equation}
0\to T_Z\to T_X|_Z\to N_{Z/X}\to 0.
\end{equation} 

Twisting by $O_X(-kY)$ we obtain 

\begin{equation}\label{eq:III}
0\to T_Z(-kY)\to T_X(-kY)\to N_{Z/X}(-kY)\to 0.
\end{equation} 

We can dualize (\ref{eq:II}) and (\ref{eq:III}) to 
\begin{equation}\label{eq:II*}
0\to N_{Z/X}^*\to \Omega_X^1(\log Y)|_Z\to \Omega_Z^1(\log Y)\to 0,
\end{equation}
and 

\begin{equation}\label{eq:III*}
0\to N_{Z/X}^*(kY)\to \Omega_X^1(kY)|_Z\to \Omega_Z^1(kY)\to 0.
\end{equation}

Following Griffiths' work in the absolute case, we can define a residue operator. Recall that $\dim X=n, \ \dim Z=n-q$. 

Define
\begin{equation}\label{eq:IV}
\Omega_X^{n-q+1}(\log Y)|_Z\to \Omega_Z^{n-q}(\log Y)\otimes N_{Z/X}^*\to 0 
\end{equation}
in the following way: let $\tau_1, \ldots, \tau_{n-q}$ be tangent vectors tangent to $Y$, and let $\eta$ be a normal vector to $z_0$. By (\ref{eq:II}), we can lift $\eta$ to a vector field of $T_X(-\log Y)|_Z$. This defines (\ref{eq:IV}). 

Taking cohomology and restricting to Z, we obtain a part of diagram (\ref{diagram:D}):
\begin{equation}\label{diagram:rD1}
\begin{tikzcd}
H^{n-q}(X,\Omega_X^{n-q+1}(\log Y))\ar[r]& H^{n-q}(Z,\Omega_Z^{n-q}(\log Y)\otimes N_{Z/X}^*)\ar[d, "\cong"]\\
& H^0(Z,N_{Z/X}(-Y))^*.
\end{tikzcd}
\end{equation}

Similarly, using (\ref{eq:III}), we obtain a residue

\begin{equation}\label{eq:kRes}
\Omega_X^{n-q+1}(kY)|_Z\to \Omega_Z^{n-q}(k Y)\otimes N_{Z/X}^*(kY)\to 0,
\end{equation}

which gives

\begin{equation}
\begin{tikzcd}
H^{n-q}(X,\Omega_{X}^{n-q+1}(kY)) \ar[r] & H^{n-q}(Z,\Omega_Z^{n-q}(kY)\otimes N_{Z/X}^*(kY)) \ar[d,"\cong"]\\
& H^0(Z, N_{Z/X}(-2kY))^*.
\end{tikzcd}
\end{equation}
\subsection{Kodaira-Spencer type map}
Finally, to obtain the Kodaira-Spencer map $\rho$, we can proceed as follows: 

Define the following locus of $\Hilb_X$.

\begin{equation}
\Hilb_X^{(k)}(Y, Z\cap kY):=\{[Z']\in \Hilb_X\mid Z'\cap kY = Z\cap kY \text{ as subscheme}\}
\end{equation}
then the tangent space at $[Z]\in \Hilb_X^{(k)}$ is given by 
$H^0(Z,N_{Z/X}(-kY)).$

Next, consider a map $D\to \Hilb_X^{(k)}$ and look at the codifferential.

\section{Dualities for threefolds}\label{S:Dualities}
\subsection{Motivation}\label{ss:DualMotivation}
Let $X$ be a smooth projective Calabi-Yau threefold. Consider $C\subset X$ a smooth curve. Then, using adjunction and $K_X\cong \mathscr{O}_X$, we have that 
\begin{equation}\label{eq:NormalProjective}
N_{C/X}\cong K_C\otimes N_{C/X}^* 
\end{equation}
from this, it follows that 
\begin{equation}\label{eq:dualityProj}
H^0(N_{C/X})\cong H^1(N_{C/X})^*.
\end{equation}

Now, let us consider problem 2.5 of \cite{AK91}. 
\begin{prob} Let $X$ be a Calabi-Yau threefold, and let 
$$\begin{tikzcd} \mathscr{C} \ar[r,"f"]\ar[d,"\pi"] & X \\
S&
\end{tikzcd} $$
be a family of (rational) curves parametrized by the smooth variety $S$. We have an Abel-Jacobi map $\AJ: \Alb(S)\to J(X)$. Now, if $\nu \in H^1(T_X)$, the curves that deform in the direction of $\nu$ are precisely the curves $\mathscr{C}_s$ for which $\dAJ^*(\nu)(s)=0$.
\end{prob}

\noindent The solution presented next is due to Mark Green. 

\begin{prop}\label{prop:dAJ} The dual of \begin{equation}
\dAJ: H^0(N_{C/X})\to H^{2,1}(X)^*
\end{equation}
is given by the obstruction map
\begin{equation}
H^{1,2}(X)\cong H^1(T_X)\to H^1(N_{C/X}).
\end{equation}
\end{prop}
\begin{proof}
The dual of $\dAJ$ is defined by:
$$\Omega_X^2|_C\cong \Omega_C^1\otimes N_{C/X}^*. $$
This is what we used in Section \ref{S:infAJ}.

Taking cohomology, we obtain 
$$H^1(\Omega_X^2)\to H^1(\Omega_X^2|_C)\to H^1(\Omega_C^1\otimes N_{C/X}^*) .$$
Now, using on the RHS Serre duality and on the LHS equation (\ref{eq:NormalProjective}), we obtain that 
$$H^1(N_{C/X})\cong H^1(\Omega_X^1\otimes N_{C/X}^*)\cong H^0(N_{C/X})^* .$$
The obstruction map is given by
$$\begin{tikzcd} T_X \ar[r]\ar[rd] & T_X|_C \ar[d]\\
& N_{C/X}
\end{tikzcd}, $$
inducing $H^1(T_X)\to H^1(N_{C/X})$. Now, if $\omega\in H^0(K_X)$ is a generator, then the following diagram commutes
$$\begin{tikzcd} T_X\ar[r]\ar[d,"\omega"] & N_{C/X}\ar[d,"by \ (\ref{eq:NormalProjective})"]\\
\Omega_X^2 \ar[r] & N_{C/X}^*\otimes K_C
\end{tikzcd}.  $$
\end{proof}

We obtain as a corollary.
\begin{cor}\label{cor:UnobsByAJ} The pair $(X,C)$ is unobstructed in the direction $\nu\in H^1(T_X)$ $\iff$ $$\nu\perp AJ_X(H^0(N_{C/X}))\iff AJ_X^*\nu =0$$
\end{cor}

\subsection{Formulas for the twisted Tangent space}\label{ss:FormulasTw}
Next, we want to generalize Proposition \ref{prop:dAJ} to the pair setting. 
Let $X$ be smooth projective threefold and assume there exists $Y\in \abs{-\frac{1}{m}K_X}$ smooth.

\begin{lem}\label{lem:TangIsos} Using the fact that $K_X(mY)\cong \mathscr{O}_X$, we have the following formulas with $k\in \mathbb{Z}$
\begin{align}
&T_X= \Omega_X^2(mY) \label{form:T} \\
&T_X(-\log Y)=\Omega_X^2(\log Y)(m-1)Y \label{form:Tlog}\\
&T_X(-kY)=\Omega_X^2(m-k)Y\label{form:TTwis}\\
&T_X(-\log Y)(-kY)=\Omega_X^2(\log Y)(m-k-1)Y
\end{align}
\end{lem}
\begin{proof}
Let us prove (\ref{form:Tlog}). The others are proved similarly or obtained by twisting. 
By contraction, we have that
$$T_X(-\log Y)\otimes K_X(\log Y)\overset{\sim}\to \Omega_X^2(\log Y).$$
As $K_X(Y)\cong K_X(\log Y)$ and using $K_X(mY)\cong \mathscr{O}_X$, we have that 
$$ T_X(-\log Y)\otimes \mathscr{O}_X \cong T_X(-\log Y)\otimes K_X(\log Y)((m-1)Y) ,$$
and using the contraction map gives the desired result.
\end{proof}

\subsection{Duality of normal bundle in the log-case}\label{ss:DualityNormalLog}
Here, we want to generalize the duality formula (\ref{eq:dualityProj}) to the pair case. 
Let $X$ be smooth projective threefold and assume there exists $Y\in \abs{-\frac{1}{m}K_X}$ smooth. Consider $C\subset X$ a smooth curve and let $k$ be an integer. 

\begin{thm}\label{thm:duality}
We have the following duality:
$$H^0(N_{C/X}((k-m)Y))\cong H^1(N_{C/X}(-kY))^* $$
\end{thm}
\begin{proof}
By Serre duality we have
$$H^1(N_C(-kY))^*\cong H^0(K_C\otimes N_{C/X}^*(kY)). $$
By adjunction and because $Y\in \abs{-\frac{1}{m}K_X}$
$$K_C=(-K_X)|_C\otimes \det N_{C/X}=\Oo_C(-mY)\otimes \det N_{C/X}. $$
Using
$$\det N_{C/X}\otimes N_{C/X}^*\cong N_{C/X}$$
and adjunction, we conclude that 
$$H^0(K_C\otimes N_{C/X}^*(kY))\cong H^0(N_{C/X}((k-m)Y). $$
\end{proof}

\begin{cor}\label{cor:dual2-1} For $m=2$ and $k=1$, we have that
$$H^0(N_{C/X}(-Y))\cong H^1(N_{C/X}(-Y))^* $$
\end{cor}

\begin{rem}\label{rem:dual} It follows that for $X$ a smooth Fano threefold, the only possible values for $m$ and $k$ to have 
$$H^0(N_{C/X}((k-m)Y))^*=H^1(N_{C/X}((k-m)Y))=H^1(N_{C/X}(-kY)), $$ 
this is, having the duality of section \ref{ss:DualityNormalLog}, are
\begin{itemize}
\item $m=k=0$, which we think as the classical Calabi-Yau case, 
\item $m=2, k=1$, which is the case of Corollary \ref{cor:dual2-1}, and
\item $m=4,k=2$, in which case $X=\mathbb{P}^3$ and $Y\cong \mathbb{P}^2$.
\end{itemize}
\end{rem}

\subsection{Obstruction maps and Abel-Jacobi maps}\label{ss:ObsAJ}
Consider a smooth curve $C\subset X$. For simplicity, assume it is transversal to $Y$. Every twist of the tangent bundle as in Lemma \ref{lem:TangIsos} give rise to a twist of the surjections $T_X\to N_{C/X}$ or $T_X(-\log Y)\to N_{C/X}$. In particular, we have obstruction maps given by taking the first cohomology groups of these twists. 

On the other hand, the isomorphisms of Lemma \ref{lem:TangIsos} with a twist of $\Omega_X^2$ give rise to a series of infinitesimal $\AJ$ morphisms defined as in Section \ref{S:infAJ}.

The main point of this subsection is to show that the analogues of Proposition \ref{prop:dAJ} and Corollary \ref{cor:UnobsByAJ} hold in every case of Lemma \ref{lem:TangIsos}.

\begin{prop} We have that the following diagram is commutative:
$$
\begin{tikzcd} T_X(-\log Y)(-kY) \ar[r] \ar[d,"\cong"] & N_{C/X}(-kY) \ar[d,"\cong"]\ar[r] & 0 \\
\Omega_X^2(\log Y)(m-k-1)Y \ar[r] & K_C((m-k)Y)\otimes N_{C/X}^*\ar[r] & 0
\end{tikzcd}.$$
\end{prop}
\begin{proof}
Using adjunction, we find that 
$$N_{C/X}(-kY)\cong K_C((m-k)Y)\otimes N_{C/X}^* .$$

The isomorphism 
$$ T_X(-\log Y)(-kY) \cong \Omega_X^2(\log Y)(m-k-1)Y, $$
follows from Lemma \ref{lem:TangIsos} above. 

Finally, the arrow 
$$ \Omega_X^2(\log Y)(m-k-1)Y \to K_C((m-k)Y)\otimes N_{C/X}^*,$$
is given by restriction.
\end{proof}

\begin{cor}\label{cor:Obs=AJlogY} The dual of the infinitesimal log $\AJ$ map
$$\dAJ^*:H^1(\Omega_X^2(\log Y)((m-k-1)Y)) \to H^0(N_{C/X}((k-m)Y))^*, $$
is naturally identified with the obstruction map
$$H^1(T_X(-\log Y)(-kY))\to H^1(N_{C/X}(-kY)). $$
\end{cor}
\begin{proof}
It follows by using Serre duality that $$H^0(N_{C/X}((k-m)Y))^*\cong H^1(N_{C/X}(-kY)).$$
\end{proof}

Now, for (\ref{form:T}) in Lemma \ref{lem:TangIsos}, we have the following.
\begin{prop} The following diagram is commutative
\begin{equation}
\begin{tikzcd}[column sep=huge, row sep=small]
T_X\big|_C \ar[r] \ar[d,"\simeq"'] & N_{C/X} \ar[r] \ar[d,"\cong"] & 0\\
\Omega^2_X(mY)\big|_C \ar[r] & K_C(mY)\otimes N_{C/X}^* \ar[r] & 0
\end{tikzcd}
\end{equation}
\end{prop}

We also have the analogue of Corollary \ref{cor:Obs=AJlogY} that we omit. There are similar propositions and corollaries for any remaining cases of Lemma \ref{lem:TangIsos}.

\begin{exmp}
Let us consider a computation as done in Proposition 2.1 of \cite[Proposition 2.1]{AK91}. The setting is as follows:
let 
$$ X=\{F=0\}\subset\mathbb{P}^4,\qquad F=z_0^3+z_1^3+z_2^3+z_3^3+z_4^3
$$
be the Fermat cubic and let 
$$ p=[1:-1:0:0:0]\in X.$$
The elliptic plane curve defined by 
$$E:=\{z_2^3+z_3^3+z_4^3=0\}\subset\mathbb{P}^2$$
is contained in $X$ and determines a cone with vertex $p$.

Fix a generic hyperplane $H=\{\,\ell=0\,\}$ with passing through $p$, and let $z=[0:0:a:b:c]\in E$ be a point of the base cubic with $a+b+c=0$. Consider the line $L$  passing through $p$ and $z$. 

We want to consider first order deformations of $X$ preserving the hyperplane $H$. This is represented by
\[
F+t\,\ell\cdot Q\quad(\bmod t^2),
\]
with $Q$ a homogeneous quadratic.

 Let
\[
R_2=\bigl(\mathbb{C}[z_0,\dots,z_4]/(\partial F)\bigr)_{2}\cong H^1(T_X(-H))
\]
be the degree-$2$ part of the Jacobian ring. This can be identifies with first order deformations of $X$ preserving $H$.

We want to study the image of $H^1(T_X(-H))\to H^1(N_{L/X}(-H))$. To compute the RHS, we use:
$$0\to N_{L/X}\to N_{L/\mathbb{P}^4}\to N_{X/\mathbb{P}^4}|_L\to 0. $$
We know that $N_{X/\mathbb{P}^4}|_L\cong \Oo_L(3)$, $N_{L/\mathbb{P}^4}\cong \oplus \Oo_L(1)$ and twisting by $-H$ we obtain that 
$$H^1(N_{L/X}(-H))\cong H^0(\Oo_L(2))/ \Ima H^0(\oplus \Oo_L). $$

\begin{prop}
 Then the natural restriction map
\[
r:\ H^1\bigl(T_X(-H)\bigr)\cong R_2 \longrightarrow H^1\bigl(N_{L/X}(-H)\bigr)
\]
is given on the representative quadratic
\[
Q=\sum_{0\le i<j\le4} q_{ij}\,z_i z_j
\qquad(\text{taken modulo }z_i^2)
\]
by the formula
\[
r([Q]) \;=\; \bigl(a(q_{02}-q_{12})+b(q_{03}-q_{13})+c(q_{04}-q_{14})\bigr)\,[uv]
\in H^1\bigl(N_{L/X}(-H)\bigr),
\]
where $[uv]$ denotes the class of the monomial $uv\in H^0(\Oo_L(2))$ under the identification 
$$H^1(N_{L/X}(-H))\cong H^0(\Oo_L(2))/\Ima H^0(\oplus \Oo_L)\cong\mathbb{C}\cdot[uv].$$ In particular, for 
$z=(0:0:1:-1:0)$ one has
\[
r([Q])=(q_{02}-q_{03}-q_{12}+q_{13})\,[uv],
\]
so $r$ is a nonzero linear functional on $R_2$.

Therefore a generic first order deformation of $X$ preserving $H$ produces a nonzero obstruction in $H^1(N_{L/X}(-H))$, and the line $L$ does not deform (to first order) with a generic such deformation.
\end{prop}

\begin{proof}
Parametrise the line $L$ joining $p=[1:-1:0:0:0]$ and $z=[0:0:a:b:c]$ by
\[
\varphi:[u:v]\longmapsto (u,-u, a v, b v, c v).
\]

 The Kodaira–Spencer class of this deformation is the class of $Q$ in $R_2\cong H^1(T_X(-H))$, and the restriction map $r$ is induced by restricting $Q$ to $L$ and then projecting to the cokernel
\[
H^0\bigl(\Oo_L(2)\bigr)\longrightarrow H^1\bigl(N_{L/X}(-H)\bigr)\cong
\frac{H^0(\Oo_L(2))}{\Ima\gamma},
\]
where $\gamma:H^0(N_{L/\mathbb{P}^4}(-1))\to H^0(\Oo_L(2))$ is the map sending a triple of constant sections $(c_0,c_1,c_2)$ to $\sum_{i=0}^4 c_i z_i(\varphi)^2$.

Compute $Q|_L$ by substituting $z_0=u,\ z_1=-u,\ z_2=av,\ z_3=bv,\ z_4=cv$ into each monomial $z_i z_j$. One obtains
\[
Q|_L \;=\; -q_{01}\,u^2 \;+\; u v\bigl(a(q_{02}-q_{12})+b(q_{03}-q_{13})+c(q_{04}-q_{14})\bigr)
\;+\; v^2\bigl(ab\,q_{23}+ac\,q_{24}+bc\,q_{34}\bigr).
\]
On the other hand
\[
z_0(\varphi)^2=z_1(\varphi)^2=u^2,\qquad z_2(\varphi)^2=a^2v^2,\ z_3(\varphi)^2=b^2v^2,\ z_4(\varphi)^2=c^2v^2,
\]
so the image $\Ima\gamma$ contains the subspace spanned by $u^2$ and $v^2$. Therefore the cokernel $H^0(\Oo_L(2))/\Ima\gamma$ is one–dimensional and may be represented by the class of $uv$. Reducing $Q|_L$ modulo $\Ima\gamma$ kills the $u^2$ and $v^2$ terms and leaves only the $uv$–coefficient. This yields the asserted formula for $r([Q])$.

For the special point $z=(0:0:1:-1:0)$ we have $(a,b,c)=(1,-1,0)$, so
\[
r([Q])=(q_{02}-q_{03}-q_{12}+q_{13})\,[uv].
\]
The linear functional $Q\mapsto q_{02}-q_{03}-q_{12}+q_{13}$ is nontrivial on $R_2$, hence its kernel has codimension $1$. Thus a general $[Q]\in R_2$ maps to a nonzero class in $H^1(N_{L/X}(-H))$, and the corresponding first order deformation of $X$ fixing $H$ obstructs the deformation of $L$.
\end{proof}

\begin{cor} For $X,L,H$ as above, the infinitesimal Abel-Jacobi
$$H^0(N_{L/X}(-H))\to H^1(\Omega_X^2(H)) $$
is injective.
\end{cor}
\begin{proof}
This follows using (\ref{form:TTwis}), its respective Corollary of the form \ref{cor:Obs=AJlogY} and the above computations.
\end{proof}

\end{exmp}

\subsection{Twisted vanishings for Fano pairs}
Let \(X\) be a smooth projective Fano \(n\)-fold and let \(Y\subset X\) be a smooth divisor.
Assume
\[
K_X + mY \simeq \Oo_X
\]
for some integer \(m\ge 1\). Fix an integer \(k\) and set
\[
\begin{aligned}
S_0 &:= T_X \simeq \Omega_X^{\,n-1}(mY),\\
S_1 &:= T_X(-\log Y) \simeq \Omega_X^{\,n-1}(\log Y)((m-1)Y),\\
S_2 &:= T_X(-kY) \simeq \Omega_X^{\,n-1}((m-k)Y),\\
S_3 &:= T_X(-\log Y)(-kY) \simeq \Omega_X^{\,n-1}(\log Y)((m-k-1)Y),
\end{aligned}
\]
and denote \(s_0:=m,\ s_1:=m-1,\ s_2:=m-k,\ s_3:=m-k-1\).

\begin{thm}\label{thm:vanishing}
Suppose \(X\) is Fano (so \(-K_X\) is ample). Then:

\medskip\noindent (A) If \(s_j\ge 1\), (i.e. \(\Oo_X(s_jY)\) is ample), then
\[
H^q(X,S_j)=0\qquad\text{for every }q>1.
\]

\medskip\noindent (B) If one of the equalities
\[
s_1=0\ (\text{i.e. }m=1),\qquad s_2=0\ (\text{i.e. }m=k),\qquad s_3=0\ (\text{i.e. }m=k+1)
\]
holds, then for the corresponding \(j\in\{1,2,3\}\)
\[
H^q(X,S_j)=0\qquad\text{for }q=2,\dots,n-1.
\]

In particular, in all the cases above \(H^2(X,S_j)=0\).
\end{thm}

\begin{proof} 
\medskip\noindent\textbf{(A).} Each \(S_j\) is of the form
\(\Omega_X^{\,n-1}(\log Y)\otimes\Oo_X(s_jY)\). Since \(X\) is Fano and \(s_j\ge1\),
the line bundle \(\Oo_X(s_jY)\) is ample, hence Akizuki--Nakano (and Esnault--Viehweg \cite{Esnault-Viehweg}
in the logarithmic instances) give
\[
H^q\big(X,\Omega_X^{\,n-1}(\text{(log)})\otimes\Oo_X(s_jY)\big)=0\quad(q>1),
\]
which is precisely \(H^q(X,S_j)=0\) for \(q>1\).

\medskip\noindent\textbf{(B).}
Adjunction gives
\[
K_Y = (K_X+Y)|_Y = \Oo_Y(-(m-1)Y),\qquad N_{Y/X}\simeq\Oo_Y(Y).
\]
Note that \(\Oo_Y(Y)\) is ample because \(-K_X\) is a positive multiple of \(Y\).

\smallskip\noindent\emph{Case \(s_1=0\) (\(m=1\)).} From
\[
0\to S_1=T_X(-\log Y)\to T_X\to N_{Y/X}\to0
\]
the middle term \(T_X\cong\Omega_X^{\,n-1}(mY)\) has a twist by $\Oo_X(mY)$ and so
vanishes in cohomology for degrees \(>1\) by (A). We have that
\(N_{Y/X}\simeq\Oo_Y(Y)=K_Y\otimes(-K_X)|_Y\), as $-K_X$ is ample, Kodaira vanishing on \(Y\) yields \(H^i(Y,N_{Y/X})=0\) for $i>0$. The long exact sequence gives the claim for \(S_1\).

\smallskip\noindent\emph{Case \(s_2=0\) (\(m=k\)).} From
\[
0\to S_2=T_X(-kY)\to T_X(-(k-1)Y)\to N_{Y/X}(-(k-1)Y)\to0
\]
the middle term is \(T_X(-(k-1)Y)\cong\Omega_X^{\,n-1}(Y)\) and
vanishes in cohomology in degrees \(>1\) by (A). The quotient is
\[
N_{Y/X}(-(k-1)Y)\simeq \Oo_Y((1-(k-1))Y)=K_Y\otimes\Oo_Y(Y),
\]
and \(\Oo_Y(Y)\) is ample; Kodaira on \(Y\) gives \(H^i(Y,K_Y\otimes\Oo_Y(Y))=0\) for $i>0$.
The long exact sequence yields \(H^q(X,S_2)=0\) for \(q=2,\dots,n-1\).

\smallskip\noindent\emph{Case \(s_3=0\) (\(m=k+1\)).} From
\[
0\to S_3=T_X(-\log Y)(-kY)\to S_2=T_X(-kY)\to N_{Y/X}(-kY)\to0
\]
the middle term \(S_2\) satisfies $T_X(-kY)\cong \Omega_X^{n-1}(Y)$ and so vanishes in degrees \(>1\) by (A).
The quotient is
\[
N_{Y/X}(-kY)\simeq \Oo_Y((1-k)Y)=K_Y\otimes\Oo_Y(Y),
\]
again with \(\Oo_Y(Y)\) ample; Kodaira on \(Y\) gives \(H^i(Y,K_Y\otimes\Oo_Y(Y))=0\) for $i>0$.
Hence \(H^q(X,S_3)=0\) for \(q=2,\dots,n-1\).

This completes the proof.
\end{proof}

\begin{cor}\label{cor:moduli-smoothness}
Let $(X,Y)$ be a pair satisfying the assumptions of Theorem~\ref{thm:vanishing}. Then:
Assuming $k\ge 0$, we have natural inclusions $S_j \hookrightarrow T_X$ for all $j=1,2,3$. Consequently,
    \[
    H^0(X, T_X) = 0 \implies H^0(X, S_j) = 0.
    \]
    Thus, if $X$ has no infinitesimal automorphisms, then neither do the pairs (with or without higher order constraints).
\end{cor}

\subsection{Appendix: rational curves of arbitrarily high degree in the cubic threefold}
Below, we present two ways to construct smooth rational curves of arbitrarily high degree on cubic threefolds. The first works for general cubics and the second for cubics containing a plane. All of these constructions are inspired by Clemens' methods in \cite{C83}.

One might hope to show that $H^0(N_{C/X}(-H))=0$ with $H$ a hyperplane section, and this would imply that $N_{C/X}$ is maximally balanced (this is $N_{C/X}=\Oo(d-1)\oplus \Oo(d-1)$). This would give examples of log-rigid curves, which would have to move with an appropriate deformation of $X$ from the results above. We tried to follow the arguments of \cite[Lemma 1.15]{C83}, but they break very quickly.

\subsubsection{Construction for general cubic threefold but restricted anticanonical K3}
Let us consider $\mathbb P^4$ with homogeneous coordinates $[x_0:\ldots:x_4]$. Let $X\subset\mathbb P^4$ be a cubic threefold and let
\[
L_0,L_1\subset X
\]
be two lines meeting in a point $p=L_0\cap L_1$. We can suppose that the pair $(L_0,L_1)$ is generic.  Fix a rank--$3$ quadric cone
\[
Q:=\{x_0x_1-x_2^2=0\}\subset\mathbb P^4
\]
with ruling $\{\Pi_t\}$ and vertex the line $V=\{x_0=x_1=x_2=0\}$, and choose a hyperplane $H\subset\mathbb P^4$ such that:
\begin{enumerate}
\item $H$ does \emph{not} contain the vertex line $V$ (so $H\cap V$ is a single point $v_0$),
\item $L_0,L_1\subset H$, their intersection point $p\neq v_0$, and neither $L_0$ nor $L_1$ equals any line of the pencil $\ell_t:=\Pi_t\cap H$.
\end{enumerate}

\begin{prop}\label{prop:two-lines-sections}
Let $X, Q$ be as above. Then, for a generic cubic $X$ in the linear system of cubics containing $L_0\cup L_1$,
\begin{itemize}
\item  the intersection
\[
S:=Q\cap X
\]
has only finitely many ADE singularities supported on $V\cap X$, its minimal crepant resolution $\widetilde S$ is a K3 surface, 
\item the strict transforms $\widetilde L_0,\widetilde L_1\subset\widetilde S$ are sections of the elliptic fibration induced by the ruling of $Q$, and 
\item for a generic choice one of these sections is non-torsion in the Mordell--Weil group. 
\end{itemize}

Therefore, the multiples of that non-torsion section give infinitely many distinct sections. Pushing them down to $X$ yields smooth rational curves of arbitrarily large hyperplane degree.
\end{prop}

\begin{proof}
Let us proceed in three steps.

First, let us construct an elliptic surface inside the cubic $X$: because $Q$ is the cone over a smooth conic, it is ruled by planes $\{\Pi_t\}_{t\in\mathbb P^1}$. If $H$ does not contain $V$ then for each $t$ the intersection $\ell_t:=\Pi_t\cap H$ is a line in the pencil of lines in $H$ passing through $v_0=H\cap V$. We choose the cubic $X$ generic such that $v_0\not\in X$, then the intersection $S:=Q\cap X$ meets a general plane $\Pi_t$ in a smooth plane cubic. Thus, after resolving the ordinary singularities of $K$ which occur at points of $V\cap X$, which are generically Du Val $A_1$ singularities by the equation of the cone $Q$, the minimal crepant resolution $\widetilde S$ is a smooth K3 and the family $\{\Pi_t\cap X\}$ yields an elliptic fibration
\(\varphi:\widetilde S\to\mathbb P^1.\)

Then, we show that the two lines give sections: by hypothesis $L_0,L_1\subset H$ and neither equals some $\ell_t$, and moreover $p=L_0\cap L_1\neq v_0$. Fix a general parameter $t$. Since $\ell_t\subset\Pi_t$, the intersection $L_i\cap\Pi_t$ equals $L_i\cap\ell_t$, which is a single point for each $i=0,1$. Thus for general $t$ each $L_i$ meets the fibre $\Pi_t\cap X$ in exactly one point, and after resolution the strict transforms $\widetilde L_i$ meet the generic fibre in exactly one point. Hence each $\widetilde L_i$ is a section of $\varphi$.

Finally, genericity implies non-torsion. In the linear system of cubics containing the two fixed lines there is a positive-dimensional family. For generic choices, $L_i\cap \Pi_t\cap X$ is not a point of finite order in $\Pi_t\cap X$. Therefore a generic cubic in the family yields a non-torsion section $s$ (take $s:=\widetilde L_1$ for instance).
We conclude as Clemens by noting that multiples of $L_i\cap \Pi_t\cap X$ extend over the singular fibres to give everywhere defined sections $L_n\subset \tilde{S}$.

On an elliptic surface the N\'eron--Tate canonical height satisfies $\widehat h(ns)=n^2\widehat h(s)$ for all $n$, and the canonical height controls intersection with an ample divisor (hence the hyperplane degree) up to bounded error. Consequently the hyperplane degree of the image in $X$ of the section $n\cdot s$ grows like $n^2$, and the $n\cdot s$ are distinct for distinct $n$. For a generic choice of the ambient cubic these images are smooth rational curves in the smooth threefold $X$. Thus $X$ contains smooth rational curves of arbitrarily large degree, as stated.
\end{proof}

\subsubsection{Cubic threefold containing a plane but less rigid quadric }

Let $X \subset \mathbb{P}^4$ be a cubic containing the plane (general between those)
\[
P = \{x_0 = x_1 = 0\}.
\]
Choose a quadric $Q \subset \mathbb{P}^4$ whose restriction to $P$ is a smooth
plane conic $C_P$. For general such $Q$, the intersection
\[
S := X \cap Q
\]
is a smooth $(2,3)$–K3 surface in $\mathbb{P}^4$. Let
\[
h := H|_S \in \Pic(S)
\]
be the hyperplane class, and let
\[
p := [C_P] \in \Pic(S).
\]

\begin{prop}
The divisor
\[
D := h - p
\]
satisfies $D^2 = 0$ and defines an elliptic fibration 
\[
f : S \to \mathbb{P}^1.
\]
Moreover, we can choose a line $\ell \subset X$ disjoint from $C_P$ and a quadric
$Q$ through $\ell$ so that $\ell \subset S$ is a section of $f$. For a generic
choice, this section is not torsion, and its
multiples give rational curves on $S$ of arbitrarily large
degree.
\end{prop}

\begin{proof}
We proceed in several steps.

First, on $S$ we have $h^2 = \deg S = 6$. The curve $C_P$ is a smooth
plane conic, hence a $(-2)$–curve on the K3, so $p^2 = -2$, and
\[
h \cdot p = \deg(C_P) = 2.
\]
From this, it follows that $D^2=0$, so $D$ is an isotropic divisor class on $S$.

Next, note that the class $D = h - p$ is primitive. Since $h$ is ample,
\[
h \cdot D = h^2 - h \cdot p = 6 - 2 = 4 > 0,
\]
so $-D$ is not effective. By Riemann--Roch on a K3 surface,
\[
\chi(\mathcal{O}_S(D)) = 2 + \frac{1}{2} D^2 = 2,
\]
hence $h^0(S, \mathcal{O}_S(D)) \ge 1$, so $|D|$ is nonempty.

Now, any fixed component of $|D|$ would be a $(-2)$–curve $E$ with $E \cdot D \le 0$.
But $E \cdot h > 0$ for every effective curve $E$, since $h$ is ample. For $p$
one computes
\[
D \cdot p = h \cdot p - p^2 = 2 - (-2) = 4 > 0.
\]
Thus no $(-2)$–curve intersects $D$ negatively; hence $D$ is nef and $|D|$, having no negative fixed components, is base-point-free.

By Saint--Donat, a primitive nef divisor $D$ with $D^2 = 0$ and $|D| \neq
\varnothing$ defines a morphism
\[
f : S \to \mathbb{P}^1
\]
whose fibers are genus-$1$ curves. Therefore $|D|$ induces an elliptic fibration
on $S$.

The cubic threefold $X$ has many lines $\ell
\subset X$. Choose $\ell$ not lying in $P$ and disjoint from $C_P$. Impose the additional condition that $Q$ contains $\ell$, which
cuts out a linear subspace in the space of quadrics; a general such $Q$ is
still smooth and gives a smooth $S$ containing both $C_P$ and $\ell$.

On $S$ we have:
\[
h \cdot \ell = \deg(\ell) = 1, \qquad
p \cdot \ell = |\ell \cap C_P| = 0,
\]
hence
\[
D \cdot \ell = (h-p)\cdot\ell = 1.
\]
So $\ell$ meets each fiber in one point, hence is a section of $f$.

For general choices, the section is non-torsion in the Mordell--Weil group, then its multiples give infinitely many distinct sections. Each such
section is a rational curve on $S$ and its degree with respect to $h$ grows
quadratically by the Shioda height formula. Thus $X$ contains rational curves
of arbitrarily large degree.

\medskip\noindent
This completes the proof.
\end{proof}
\section{Infinitesimal Abel-Jacobi}\label{S:infAJCri}
In this section, we generalize a result of Clemens \cite{C89} to the pair setting which gives a criterion for non-triviality of the infinitesimal Abel-Jacobi for curves on threefolds. The idea is to use global forms in $\mathbb{P}^4$ with poles along two hypersurfaces and the extension class of the normal sequence of the curve. Then, we apply it to the Fermat cubic threefold.
\subsection{The formula} The results in this sections are valid in more generality. To simplify the notation, we will stick to the threefold case inside $\mathbb{P}^4$.

Let $X=\{F=0\}\subset \mathbb{P}^4$ be a hypersurface of degree $d$, $Y\subset X$ a smooth divisor and let $C$ to be a non-singular curve inside $X$ transverse to $Y$. 

Recall the exact sequence (\ref{eq:II}) :

$$0\to T_C(-\log Y)\to T_X(-\log Y)|_C\to N_{C/X}\to 0 .$$

\begin{prop}\label{prop:ClemCrite} Let $i:C\to X$ denote the inclusion. Let $e$ be the extension class of the sequence:
\begin{equation}\label{eq:normal}
0\to N_{C/X}\to N_{C/\mathbb{P}^4}\to i^* N_{X/\mathbb{P}^4}\to 0
\end{equation}
Then, we have the following commutative diagram. We will write $H^0(N(-Y))$ instead of $H^0(N_{C/X}(-Y))$
$$
\begin{tikzcd}[
  column sep=small,
  row sep=normal,
  nodes={font=\small}
]
 H^0(N(-Y))\otimes H^0(\Omega_{\mathbb{P}^4}^4(2X)(\log Y)) 
    \ar[rr, "{\mathrm{res}|_X}"] 
    \ar[dd] 
 & & 
 H^0(N(-Y))\otimes H^0(\Omega_X^3(Y)\otimes N_{X/\mathbb{P}^4})
    \ar[d] \\ 
 & & 
 H^0(N(-Y))\otimes H^0(i^*(\Omega_X^3(Y)\otimes N_{X/\mathbb{P}^4})) 
    \ar[d, "e"] \\
 H^0(N(-Y))\otimes H^1(\Omega_X^{2}(\log Y))
    \ar[dr, "\dAJ "]
 & & 
 H^0(N(-Y))\otimes H^1(i^*(\Omega_X^3(Y)\otimes N_{C/X})) 
    \ar[dl, "\mu"] \\ 
 & H^1(\Omega_C^1) & 
\end{tikzcd}$$

with $\mu$ the pairing induced by cup-product and contraction, and, by abuse of notation $e$ denotes the mapping induced by cup-product with the obstruction class $e$.
\end{prop}
\begin{proof}
The arrow $\mathrm{res}|_X$ is given by restriction. We use the following Lemma to identify its target.

\begin{lem}\label{lem:rest} By abuse of notation, let us denote also by $Y$ a hypersurface in $\mathbb{P}^4$ such that when intersected with $X$ gives the surface $Y\subset X$. We assume the intersection in $\mathbb{P}^4$ is transverse. We have the following isomorphisms:
\begin{align*}
&\Omega_{\mathbb{P}^4}^4(X+Y)|_X\cong \Omega_X^3(Y)\\
& \Omega_{\mathbb{P}^4}^4(2X+Y)|_X\cong \Omega_X^3(Y)\otimes N_{X/\mathbb{P}^4}.
\end{align*}
\end{lem}

The commutativity of the diagram in the proposition, follows from commutativity of the following diagram with rows: the first one is giving by considering the closed forms and taking exterior differential. We denote by $\hat{\Omega}$ closed forms. The second row is induced by multiplication by $\frac{dF}{F}$ as indicated in the diagram.
The third is a twist of the log-normal sequence of $N_{X/\mathbb{P}^4}$. The last in the normal sequence of $C$. We have used the sequence (\ref{eq:II}) recalled above the proposition.

$$\begin{tikzcd}[
  column sep=normal,
  row sep=2em,             
  nodes={font=\normalsize} 
]
0 \ar[r] 
& \hat{\Omega}_{\mathbb{P}^4}^3(X)(\log Y) \ar[r] \ar[d] 
& {\Omega}_{\mathbb{P}^4}^3(X)(\log Y)\ar[r] \ar[d] 
& {\Omega}_{\mathbb{P}^4}^4(2X)(\log Y) \ar[r] \ar[d] 
& 0 \\
0 \ar[r] 
& \dfrac{\hat{\Omega}_{\mathbb{P}^4}^3(X)(\log Y)}{\hat{\Omega}_{\mathbb{P}^4}^3(\log Y)}  \ar[r] 
& \dfrac{{\Omega}_{\mathbb{P}^4}^3(X)(\log Y)}{{\Omega}_{\mathbb{P}^4}^3(\log Y)} \ar[r, "\wedge \frac{dF}{F}"] 
& \dfrac{\Omega_{\mathbb{P}^4}^4(2X)(\log Y)}{ \Omega_{\mathbb{P}^4}^4(X)(\log Y)}\ar[r] 
& 0 \\
0 \ar[r] 
& \begin{array}{c} T_X(-\log Y) \\ \otimes \\ \Omega_{\mathbb{P}^4}^4(X+Y) \end{array} \ar[r] \ar[d] \ar[u, "\cong"] 
& \begin{array}{c} T_{\mathbb{P}^4}(- \log Y) \\ \otimes \\ \Omega_{\mathbb{P}^4}^4(X+Y) \end{array} \ar[r] \ar[d] \ar[u, "\cong"] 
& \begin{array}{c} N_{X/\mathbb{P}^4} \\ \otimes \\ \Omega_{\mathbb{P}^4}^4(X+Y) \end{array} \ar[r] \ar[d] \ar[u, "\cong"] 
& 0 \\
0 \ar[r] 
& \begin{array}{c} N_{C/X} \\ \otimes \\ \Omega_X^3 (Y) \end{array} \ar[r]  
& \begin{array}{c} N_{C/\mathbb{P}^4} \\ \otimes \\ \Omega_X^3(Y) \end{array} \ar[r] 
& \begin{array}{c} N_{X/\mathbb{P}^4} \\ \otimes \\ \Omega_X^3(Y) \end{array}|_C \ar[r] 
& 0  
\end{tikzcd}$$

whose maps are justified by the lemmas \ref{lem:rest} and \ref{lem:contr}. The commutativity from the statement of the proposition follows by looking at the transgression maps from $H^0$ to $H^1$ in the morphism of long exact cohomology sequences associated to the morphism from the top short exact sequence in the diagram above to the bottom short exact sequence.

\begin{lem}\label{lem:contr} We have isomorphisms induced by contraction:
\begin{enumerate}
\item $$ T_X(-\log Y)\otimes \Omega_{\mathbb{P}^4}^4(X)(\log Y)\overset{\sim}\to \hat{\Omega}_{\mathbb{P}^4}^3(X)(\log Y)/ \hat{\Omega}_{\mathbb{P}^4}^3(\log Y) $$

\item $$ T_{\mathbb{P}^4}(-\log Y)\otimes \Omega_{\mathbb{P}^4}^4(X)(\log Y)\overset{\sim}\to {\Omega}_{\mathbb{P}^4}^3(X)(\log Y)/ {\Omega}_{\mathbb{P}^4}^3(\log Y) $$

\item $$ N_{X/\mathbb{P}^4}\otimes \Omega_{\mathbb{P}^4}^4(X+Y) \overset{\sim}\to {\Omega}_{\mathbb{P}^4}^4(2X+Y)/ {\Omega}_{\mathbb{P}^4}^4(X+Y) $$

\end{enumerate}

\end{lem}

\begin{proof}
Let us prove (3):
By twisting the exact sequence
$$0\to \Oo(-X)\to \Oo \to \Oo_X\to 0  $$

By $K:=\Omega_{\mathbb{P}^4}^4$, $\Oo(2X), \Oo(Y)$ and using that $\Oo_X(X)\cong N_{X/\mathbb{P}^n}$, we obtain that 
$$0\to K(X+Y)\to K(2X+Y)\to K(Y)\otimes N_{X/\mathbb{P}^4}^{\otimes 2}\to 0 $$

Note that the last sheaf is supported on $X$. Using this, adjunction and Lemma \ref{lem:rest}, we obtain that
$$\frac{K(2X+Y)}{K(X+Y)}\cong K_X(Y)\otimes N_{X/\mathbb{P}^4}\cong K_{\mathbb{P}^4}(X+Y)\otimes N_{X/\mathbb{P}^4}. $$

The others can be proved by local computations.
\end{proof}

\end{proof}
\subsection{Larger split exact sequence}

In this section, we give an enlargement of the sequence (\ref{eq:normal}) such that it is a split short exact sequence.
 
Let us consider $C'\subset \mathbb{P}^4$ a second curve such that $C\cap C'=\{p_1,\ldots,p_m\}$, $C'$ is smooth and transverse at every point $p_i$ for $i=1,\ldots, m$.
If we denote by $U=C\setminus \{p_1,\ldots, p_m\}$ and by $j:U\to C$ the inclusion, we can define $$N_{C/\mathbb{P}^4}(\ast C\cap C'):=j_*(N_{C/\mathbb{P}^4}|_U), $$
the sheaf of meromorphic section with arbitrary poles along $C\cap C'$.

Inside, this sheaf, let 
$$N_{C/\mathbb{P}^4}(\log C')\subset N_{C/\mathbb{P}^4}(\ast C\cap C') $$
denote the sheaf consisting of those sections which have at most a simple pole at each $p_j$ in the direction of $C'$. In local coordinates $(z_0,z_1,z_2,z_3)$, suppose that a point $p$ is the origin and $C$ is given locally by $\{z_1=z_2=z_3=0\}$. Hence, a local basis of $N_{C/\mathbb{P}^4}$ is given by $\partial_{z_i}$ for $i=1,2,3$. If $C'$ is given by $\{z_0=z_2=z_3=0\}$, a local section of $N_{C/\mathbb{P}^4}(\log C')$ can be expressed as $\sum_{i=1}^3 a_i(z)\partial_{z_i}$ with $a_2(z), a_3(z)$ holomorphic and $a_1(z)$ meromorphic having a simple pole at zero.

In a similar fashion, we can define $N_{C/X}(\log C')$. Note that we have an exact sequence given by residues
\begin{equation}\label{eq:residues}
0\to N_{C/X}\to N_{C/X}(\log C')\to \oplus_{j=0}^m \mathbb{C}_j\to 0
\end{equation}

\begin{prop}\label{prop:vanExtClass} Suppose that there exists a surface $S\subset \mathbb{P}^4$ such that $$S\cap X=C+C'$$ with $C, C'$ as above, $S$ smooth along $C$. Then the following sequence is split:
\begin{equation}\label{eq:split}
0\to N_{C/X}(\log C') \to N_{C/\mathbb{P}^4}(\log C') \to i^*N_{X/\mathbb{P}^4}\to 0
\end{equation}
\end{prop}
\begin{proof}
We can define a map
$$\Phi: N_{C/\mathbb{P}^4}(\log C')\to N_{S/\mathbb{P}^4}|_C $$
by restricting a normal vector field in $N_{C/\mathbb{P}^4}(\log C')$  to $S$ then mod out by the tangent bundle $T_S$. Because we are allowing only simple poles in the $C'$ direction, which lies inside $T_S$ directions on $S$, this projection is well-defined.

By considering the restriction
$$\Phi|_{N_{C/X}(\log C')}:N_{C/V}(\log C')\overset{\sim}\to N_{S/Y}|_{C'}$$
and showing that it is an isomorphism (using local coordinates), we can construct an splitting of (\ref{eq:split}) by pulling-back a generator of $i^*N_{X/\mathbb{P}^4}$ to a tangent vector to $S$.
\end{proof}
\subsection{Criteria for non-triviality of the infinitesimal AJ}

\begin{thm}\label{prop:CriteriaInfAJ} If 
\begin{enumerate}
\item there is an element of $H^0(\Omega_{\mathbb{P}^4}^4(2X)(\log Y))$ which does not vanish at $p_0$ but does vanish at the others $p_j$, and 
\item there is an element $v\in H^0(N_{C/X}(-Y))$ which wedges with a tangent vector $C'$ at $p_0$ to give a non-zero element of the geometric fiber of $\det N_{C/X}(-Y)$ at $p_0$, 
\end{enumerate}
then the infinitesimal Abel-Jacobi map is non-trivial at $C$.
\end{thm}
\begin{proof}
 In order to compute the extension class $e$ of the sequence (\ref{eq:normal}):
$$e\in \Ext_C^1(i^* N_{X/\mathbb{P}^4}, N_{C/X})\cong H^1(C,(i^* N_{X/\mathbb{P}^4})^*\otimes N_{C/X} ),$$
let us apply $\RHom_C(i^* N_{X/\mathbb{P}^4},\bullet)$ to the sequences (\ref{eq:residues}):
\begin{align*}
\ldots \to \Ext_C^0(i^* N_{X/\mathbb{P}^4}, N_{C/X}(\log C'))
 &\to  \Ext_C^0(i^* N_{X/\mathbb{P}^4}, \oplus_{j=0}^m \mathbb{C}_j) \to \\ &\Ext_C^1(i^* N_{X/\mathbb{P}^4},N_{C/X})\to 0.
\end{align*}
Where $\Ext_C^1(i^* N_{X/\mathbb{P}^4}, N_{C/X}(\log C'))=0$ by Proposition \ref{prop:vanExtClass}. This is saying that we can see the class $e$ as a class in $H^0((i^* N_{X/\mathbb{P}^4})^*\otimes (\oplus \mathbb{C}_j))$ modulo $H^0((i^* N_{X/\mathbb{P}^4})^*\otimes N(\log C'))$.

We consider the following commutative diagram. Let us write $N$ for $N_{C/X}$.
$$
\begin{tikzcd}[
  column sep=0.5em,  
  row sep=3em,       
  nodes={font=\small, inner sep=2pt} 
]
 H^0(N(-Y))\otimes H^0(\Omega_X^3(Y)\otimes N_{X/\mathbb{P}^4})\otimes H^0((i^* N_{X/\mathbb{P}^4})^*\otimes N(\log C'))
   \ar[r] \ar[d] 
 & H^0(\Omega_C^1(\sum p_j))
   \ar[d] \\ 
 H^0(N(-Y))\otimes H^0(\Omega_X^3(Y)\otimes N_{X/\mathbb{P}^4})\otimes H^0((i^* N_{X/\mathbb{P}^4})^*\otimes (\oplus \mathbb{C}_j))
   \ar[r]\ar[d] 
 & H^0(\oplus \mathbb{C}_j)
   \ar[d] \\
 H^0(N(-Y))\otimes H^0(\Omega_X^3(Y)\otimes N_{X/\mathbb{P}^4})\otimes H^1((i^* N_{X/\mathbb{P}^4})^*\otimes N)
   \ar[r]\ar[d] 
 & H^1(\Omega_C^1)
   \ar[d] \\
 H^0(N(-Y))\otimes H^0(\Omega_X^3(Y)\otimes N_{X/\mathbb{P}^4})\otimes H^1((i^* N_{X/\mathbb{P}^4})^*\otimes N(\log C'))
   \ar[r] 
 & H^1(\Omega_C^1(\sum p_j)) 
\end{tikzcd}$$

If the extension class $e\not =0$ is non-trivial, suppose that $p_0$ is a point where its restriction is non-zero. As the sum of the residues of a meromorphic differential is zero, the map $H^0(\mathbb{C}_0)\to H^1(\Omega_C^1)$ is an isomorphism. Therefore under the hypothesis and using the diagram in the proof of Proposition \ref{prop:ClemCrite}, we conclude the proof.

\end{proof}

\begin{exmp} Let us consider the Fermat cubic threefold $X=\{\sum_{i=0}^4 z_i^3=0\}\subset \mathbb{P}^4$. 

Recall that a point $p\in X$ is called an \textit{Eckardt point} if the intersection of $X$ with its tangent hyperplane $T_p X$ is a cone with vertex $p$. We will consider a line $L$ is the ruling of this cone. Its normal bundle splits as $N_{L/X}=\Oo(1)\oplus \Oo(-1)$.

Let $p=[1:-1:0:0:0]$. The tangent hyperplane at $p$ is $T_p X=\{z_0+z_1=0\}$. The intersection $X\cap T_p X$ is defined by 
$$z_0+z_1=0, \quad \text{and } z_2^3+z_3^3+z_4^3=0. $$
This is a cone with vertex $p$ over the elliptic curve $E= \{z_2^3+z_3^3+z_4^3=0\}\subset \mathbb{P}^2$. We choose the line $L$ connecting $p$ to the point $z=[0:0:1:-1:0]\in E$.

We will consider $H\subset X$ a generic hyperplane section passing through $p$. For concreteness, we can take $H=\{z_2+z_3+z_4=0\}$. 

Choose a plane $S=\left\lbrace p,z,q \right\rbrace$ spanned by $L$ and the point $q=[0:1:1:0:0]$. The restriction of the cubic form $F$ to $S$ factors as:
$$F|_S=\nu\cdot Q(\lambda,\mu,\nu) $$
where $\nu=0$ defines the line $L$. An explicit calculation shows that restricted to $L$, the quadric becomes $Q|_L=3(\lambda^2+\mu^2)$. The roots of this quadric correspond to the intersection points of the residual conic $C'$ with $L$. These points are:
$$p_0=[1:-1:i:-i:0] \quad \text{and} \quad p_1=[1:-1:-i:-i:0]. $$
The conic $C'$ intersects $L$ transversely at $p_0$ and $p_1$.

We identify $H^0(\Omega_{\mathbb{P}^4}^4(2X))\cong H^0(\Oo_{\mathbb{P}^4}(1))$. To have condition (1) of the proposition \ref{prop:CriteriaInfAJ} above, we require a linear form $l$ such that $l(p_1)=0$ and $l(p_0)\not =0$. Consider the linear form:
$$l(z)=z_0+iz_3. $$

Now, as the line $L$ belongs to a cone, we have that 
$$N_{L/X}(-H)\cong \Oo\oplus \Oo(-2). $$
We can compute a generator of $H^0(N_{L/X}(-H))$ by determining the kernel of the map $\phi:H^0(N_{L/\mathbb{P}^4}(-1))\to H^0(\Oo_L(2))$. Using the basis of normal directions $\partial_{z_0}+\partial_{z_1},\partial_{z_2}+\partial_{z_3}$ and $\partial_{z_4}$,  we find that the only direction preserving the defining equation to first order along $L$ is given by $v=\partial_{z_4}$.

To check the second condition of \ref{prop:CriteriaInfAJ}, note that, at $p_0$ the section $v$ corresponds to the vector $(0,0,0,0,1)$. The tangent to the conic $T_{C'}$ corresponds to $(0,-1+i,i,-1,0)$. The wedge product in the normal fiber is determined by the determinant of their components in the transverse directions $z_1$ and $z_4$:
$$\det \left(\begin{array}{ll}
v_{z_4} & (T_{C'})_{z_4} \\
v_{z_1} & (T_{C'})_{z_1}
\end{array} \right)=\det \left(\begin{array}{ll}
1 & 0 \\
0 & -1+i
\end{array} \right) =-1+i\not =0.  $$
Since the determinant is non-zero, the sections wedge non-trivially. 
\end{exmp}

\begin{rem} The original motivation of Clemens' work in \cite{C89} was to show the existence of a smooth rational curve $C$ embedded on a quintic threefold (smooth along $C$), such that any multiple of $C$ does not move nor does it contract - showing a contrast with the surface case. 

In the LMMP program, a smooth rational curve $C$ on a surface pair $(S,\Delta)$  moves fixing the boundary if $H^0(N_{C/S}(-\Delta|_C))\not =0$. If we follow the construction of Clemens, we obtain a smooth conic $C$ inside a cubic threefold $X$ with normal bundle $\mathscr{O}(2)+\Oo$. By a generalization of \cite[Lecture 16]{CKM88} to the relative setting, we can show that $C$ does not admit a relative contraction. However, $$H^0(N_{C/X}(-H))\cong H^1(N_{C/X}(-H))^*\not =0, $$
but we are not able to show that $C$ is obstructed. One of the main reasons is the lack of a holomorphic volume form for the pair $(X,H)$ with poles along $H$, a crucial ingredient in Clemens' arguments.
\end{rem}

\section{Normal functions and invariants}\label{S:NorFun}
The constructions in this section may be well-known to experts. We include it for completeness.
\subsection{log-normal functions}

Let $X$ be smooth projective of dimension $n$, $Y\subset X$ a smooth divisor and $Z\subset U:=X-Y$ a smooth closed codim $p$ algebraic cycle. 

Then $Z$ gives a homology class in Borel-Moore
$$[Z]^{BM}\in H_{2n-2p}^{BM}(U) $$

We have the long exact sequence of the pair $(U,U\setminus Z)$
\begin{equation}\label{eq:LES-BM}
\cdots \to H^{2p-1}(U\setminus Z, \mathbb{Q})\to H^{2p}_Z(U,\QQ)\to H^{2p}(U,\QQ)\to H^{2p}(U\setminus Z, \QQ) \to \cdots
\end{equation}

As $Z$ is smooth, we have the purity isomorphism
$$H^{m-2p}(Z)(-p)\overset{\cong}\to H_Z^m(U) $$
compatible with Hodge structures. In particular, 
\begin{equation}\label{eq:Pur0}
[Z]\in H^0(Z)(-p) \cong H_Z^{2p}(Z),
\end{equation}
where $[Z]$ denotes the class dual to $[Z]^{BM}$ via the isomorphism given by cap product with the fundamental class $[Z]^{BM}$:
\begin{equation}\label{eq:PD-BM}
H^{i}(Z)\overset{\sim}\to H_{2n-2p-i}^{BM}(Z).
\end{equation}

All the above can be refined to allow $Z$ singular. 

\begin{lem} Assume the class $[Z]^{BM}=0$ vanishes in $H_{2n-2p}^{BM}(U)$, then the class $[Z]\in H_Z^{2p}(Z)$ goes to zero under the map 
$$ H_Z^{2p}(U,\QQ)\to H^{2p}(U,\QQ). $$
\end{lem}
\begin{proof}
Composing the isomorphisms (\ref{eq:Pur0}) with (\ref{eq:PD-BM}), we obtain 
$$H_Z^{2p}(U)\overset{\sim}\to H_{2n-2p}^{BM}(Z). $$
We obtain a commutative diagram
$$\begin{tikzcd}
H_Z^{2p}(U) \ar[r, " \sim"] \ar[d] & H_{2n-2p}^{BM}(Z) \ar[d]\\
H^{2p}(U)\ar[r,"\sim"]& H_{2n-2p}^{BM}(U)
\end{tikzcd} $$
but by assumption, $[Z]^{BM}$ maps to zero  $H_{2n-2p}^{BM}(U)$ with the right vertical map. From this we obtain the claim.

\end{proof}

Now, assuming the class $[Z]^{BM}=0$ vanishes in $H_{2n-2p}^{BM}(U)$, we obtain from the long-exact sequence (\ref{eq:LES-BM}) that there is a class 
$$[Z']\in H^{2p-1}(U\setminus Z) $$
mapping to $[Z]\in H_Z^{2p}(U)$. This gives an extension 
$$0\to H^{2p-1}(U)\to E_Z\to \QQ(-p)\to 0$$
Therefore, we obtain an extension class 
$$[e_Z]\in \Ext_{MHS}^1(\QQ(-p),H^{2p-1}(U)).$$

\subsection{First construction of the infinitesimal invariant}

Let $X$ be smooth projective and $Y$ a smooth divisor as above. Consider the extension 
$$0\to H^{2p-1}(U)(p)\to E\to \mathbb{Q}(0)\to 0 $$
with the weights of $H^{2p-1}(U)(p)$ being $-1$ and $0$.

Suppose this data varies in a family $S$. Equivalently, that we have a VMHS of the form 
$$0\to \mathcal{H}\to \mathcal{E}\overset{\pi}\to \QQ_S(0)\to 0 $$
As 
\begin{align*}
F^0\QQ_S(0)=\QQ_S(0)\\
F^1\QQ_S(0)=0
\end{align*}
and $\pi$ respects the Hodge filtration, we have that $\pi$ factors via the quotient
$$\bar{\pi}:\mathcal{E}/F^1\mathcal{E}\to \QQ_S(0)  $$

In a polydisk $\Delta\subset S$, there exists a holomorphic section
$$\sigma \in H^0(\Delta, \mathcal{E}/F^1\mathcal{E})$$
lifting the constant section $1\in \QQ_S(0)$.

Pick a lift $s\in H^0(\Delta, \mathcal{E})$ of $\sigma$. We have that 
$$\nabla_{\mathcal{E}}S\in \Omega_\Delta^1\otimes \mathcal{E}. $$

As the section $1\in \QQ_S(0)$ and the Hodge filtration are constant, then $\nabla _{\mathcal{E}}s $ under the projection $\mathcal{E}\to \QQ(0)$ vanishes, we have then that 
$$\nabla_{\mathcal{E}}s\in H^0(\Delta, \Omega_{\Delta}^1\otimes \mathcal{H}). $$

To have a well-defined invariant, we have to mod-out by the ambiguities that can come fro $H^0(\Delta, \mathcal{H})$ or from $H^0(\Delta, F^1\mathcal{E})$.

We can therefore define the infinitesimal invariant 
 $$\delta \nu:= [\nabla s]\in H^0\left(\Delta, \frac{\Omega_{\Delta}^1\otimes \mathcal{H}}{\nabla (\mathcal{H})+\Omega_\Delta^1\otimes F^0 \mathcal{H}}\right) $$

\subsection{Hodge complexes}

Here, give a construction of the first infinitesimal invariant in the quasi-projective case following Voisin's book \cite{V02}.

Let us consider families $\pi_X:\mathcal{X}\to B$ and $\pi_Y:\mathcal{Y}\to B$ of smooth projective varieties with $\mathcal{Y}\subset \mathcal{X}$ a relative smooth divisor and $B$ a Stein ball. We can consider the local system $$H^p:=R^p\pi_* \mathbb{C} $$ with $\pi: \mathcal{X} \setminus \mathcal{Y} \to B$.

We have the de Rham holomorphic complex 
$$DR(H^p):= 0\to \mathcal{H}^p \overset{\nabla}\to \mathcal{H}^p \otimes \Omega_B^1 \ldots \overset{\nabla}\to \mathcal{H}^p \otimes \Omega_B^N\to 0$$
with $N=\dim B$. We can filter this complex by using Griffiths transversality
$$F^l DR(H^p):=0\to F^l\mathcal{H}^p \overset{\nabla}\to F^{l-1}\mathcal{H}^p \otimes \Omega_B^1 \ldots \overset{\nabla}\to \mathcal{H}^p \otimes F^{l-N}\Omega_B^N\to 0 . $$
We can take the graded complex 
$$Gr^l_F DR(H^p)= F^l DR(H^p)/F^{l+1}DR(H^p) $$
If we define $$\mathcal{H}^{l,r}:= F^l \mathcal{H}^p / F^{l+1}\mathcal{H}^{p+1},$$
we obtain that:
$$Gr^l_F DR(H^p)^k= \mathcal{H}^{l-k, p-l+k}\otimes \Omega_B^k .$$

\subsection{Holomorphic Leray filtration}
We have an inclusion of holomorphic vector bundles
$$\pi^* \Omega_B^\bullet \to \Omega_\mathcal{X}^\bullet(\log \mathcal{Y}), $$
with quotient $\Omega_{\mathcal{X}/B}^\bullet(\log \mathcal{Y})$ consisting of relative holomorphic log-forms. We have a Leray filtration $L^\bullet \Omega_\mathcal{X}^\bullet(\log \mathcal{Y})$ given by 
$$ L^l \Omega_\mathcal{X}^\bullet(\log \mathcal{Y})=\pi^* \Omega_B^l\otimes \Omega_\mathcal{X}^{\bullet-l}(\log \mathcal{Y}). $$
It induces a filtration $L^lR^p \pi_* \Omega_\mathcal{X}^k$ on $R^p \pi_* \Omega_\mathcal{X}^k$. This filtered complex has then a spectral sequence 
$$ E_{r}^{l,q} \Rightarrow R^{l+q}\pi_* \Omega_\mathcal{X}^k(\log \mathcal{Y}). $$

Modifying Voisin's arguments \cite[Proposition 17.9]{V02} to the log-setting, one can prove the following.
\begin{prop}\label{prop:HolLerFil} Fixing $q$, the complex $(E_1^{l,q},d_1)$ is isomorphic to the complex $(\mathcal{H}^{k-l,q+l}\otimes\Omega_B^l,\bar{\nabla)}$ with $\bar{\nabla}$ being induced in the quotients by $\nabla$.
\end{prop}
\subsection{Infinitesimal invariant}\label{ss:InfInvVoi}
Let us consider a class 
$$ \alpha \in \Ker(H^p(\mathcal{X},\Omega_{\mathcal{X}}^k(\log \mathcal{Y})))\to H^0(B, R^p \pi_* \Omega_{\mathcal{X}/B}^k (\log \mathcal{Y}))$$
that vanishes on the fibers.

As an intermediate step, we can consider the image of $\alpha$ in $H^0(B,R^p\pi_*\Omega_{\mathcal{X}}^k (\log \mathcal{Y}) )$. By abusing notation, we denote this class also by $\alpha$. In the Leray filtration, we have that 
$$L^1 R^p\pi_* \Omega_{\mathcal{X}}^k(\log \mathcal{Y})=\Ker(R^p\pi_* \Omega_{\mathcal{X}}^k(\log \mathcal{Y}) \to R^p\pi_* \Omega_{\mathcal{X}/B}^k(\log \mathcal{Y})) $$
Therefore, the class $\alpha$ lives in $H^0(B, L^1 R^p\pi_* \Omega_{\mathcal{X}}^k(\log \mathcal{Y}))$ and admits a class in 
$$H^0(B, L^1 R^p\pi_* \Omega_{\mathcal{X}}^k(\log \mathcal{Y})/ L^2 R^p\pi_* \Omega_{\mathcal{X}}^k(\log \mathcal{Y})) $$
but 
$$L^1 R^p\pi_* \Omega_{\mathcal{X}}^k(\log \mathcal{Y})/ L^2 R^p\pi_* \Omega_{\mathcal{X}}^k(\log \mathcal{Y}) \cong E_{\infty}^{1,p-1}$$

By degree reasons, none of $d_r$, $r\geq 2$, can map to $E_r^{1,p-1}$. Therefore, we have that $E_\infty^{1,p-1}\subset E_{2}^{1,p-1}$.

Therefore the class of $\delta \alpha$ in $E_{2}^{1,p-1}$ is the first infinitesimal invariant.

\begin{rem} Under hypothesis such that  
$$ L^1 R^p\pi_* \Omega_{\mathcal{X}}^k(\log \mathcal{Y}) = L^{l_0} R^p\pi_* \Omega_{\mathcal{X}}^k(\log \mathcal{Y}) $$ for some $l_0$, we can define higher invariants living in $E_{\infty}^{l_0,p-l_0}$. See \cite[Section 17.2.2]{V02}
\end{rem}

\subsection{Example}
As a simple application, let us compute the analogue of Green's third infinitesimal invariant for a smooth log-CY 3-fold in the following setting. Let $X_0$ be a smooth projective 3 and $Y_0\in \abs{-K_X}$ a smooth anticanonical divisor. As deformations of the pair $(X_0,Y_0)$ are unobstructed \cite[Prop. 2.5]{AGG24}, there exists a universal family $(\mathcal{X},\mathcal{Y})\to B$ with the fiber over a base point $0\in B$ equal to $(X_0,Y_0)$. 

\begin{prop} Take $\alpha\in H^0(\Omega_B^1\otimes \Omega_{\mathcal{X}}^3)$ then it satisfies the conditions of Section \ref{ss:InfInvVoi} and it has associated a class $\delta \alpha \in E_2^{1,p-1}$ with $p=0$ and $k=4$. Suppose that $\dim B\geq 2$ then $\delta \alpha=0$.
\end{prop}
\begin{proof}
We will prove in fact that $$E_2^{1,-1}\cong \ker\left (\Omega_B^1\otimes \mathcal{H}^{3,0}\to \Omega_B^2\otimes \mathcal{H}^{2,1} \right )=0.$$
Recall that, by definition and by Proposition \ref{prop:HolLerFil}, we have
$$E_1^{l,q}\cong \Omega_B^l\otimes R^{l+q}\pi_*\Omega^{k-l}_{\mathcal{X}/\mathcal{B}}. $$
Letting $l=1, q=-1, k=4$ gives 
$$E_1^{1,-1}\cong \Omega_B^1\otimes R^0\pi_*\Omega^3_{\mathcal{X}/\mathcal{B}}\cong \Omega_B^1\otimes \mathcal{H}^{3,0}.$$
Using the component of Gauss-Manin:
$$\bar{\nabla}:\Omega_B^1\otimes \mathcal{H}^{3,0} \to \Omega_B^2\otimes \mathcal{H}^{2,1}. $$
Hence
$$E_2^{1,-1}=\ker (\Omega_B^1\otimes \mathcal{H}^{3,0}\to \Omega_B^2\otimes \mathcal{H}^{2,1}), $$
because $E_1^{0,-1}=0$.

Now, infinitesimal Torelli gives an isomorphism:
$$T_B\overset{\sim}\to \Hom(\mathcal{H}^{3,0}, \mathcal{H}^{2,1}) $$
dualizing and using the fact that $\dim B\geq 2$, we obtain the result.
\end{proof}

\section{Log connectivity}\label{S:logConnect}
\subsection{Zariski theorem of Lefschetz type}
Let $X$ be a smooth projective variety of dimension $n+1$. For $L\to X$ an ample divisor let $S\subset \mathbb{P}(H^0(X,L))$ be the locus of smooth hyperplane sections and let $Y_s$ with $s\in S$.

The Lefschetz hyperplane Theorem tell us that for generic $s$ 
$$H^i(X,Y_s,\mathbb{Q})=0 \quad \text{for } i< n $$

Now, let us consider a general hyperplane $D$ (or a family of them $H_t$). We have a Zariski theorem of Lefschetz type that says

$$H^i(X\setminus D,Y_s\setminus D_s)=0 \quad \text{ for } i< n $$
where we denote $D\cap Y_s$ by $D_s$. We will give a Hodge theoretic proof of this.

\begin{thm}[Zariski thm of Lefschetz type]\label{thm:ZarLef}
Let $X$ be smooth projective of dimension $n+1$. Consider $Y,D$ hypersurfaces with $D$ SNC and $Y$ general and ample.
Consider $X^\circ=X\setminus H, Y^\circ=Y\setminus (H\cap Y)$. Then the map
$$H^i(X^\circ, \mathbb{Q})\to H^i(Y^\circ, \mathbb{Q}) $$
induced by the inclusion $i^\circ:Y^\circ \hookrightarrow X^\circ$
is an isomorphism for $i\leq n-1$ and injective for $i=n$
\end{thm}
\begin{proof}
It suffices to work over $\mathbb{C}$. We will prove the injectivity at the level of $\Gr_F^p$.

By degeneration of Hodge-de Rham spectral sequence at $E_1$ we have that 
$$\Gr_F^p H^i(X^\circ, \mathbb{C})\cong H^{i-p}(X,\Omega_X^p(\log D))$$
Hence, if we can show that 
\begin{equation}\label{eq:inj}
\Gr_F^p H^i(X^\circ, \mathbb{C})\hookrightarrow \Gr_F^p H^i(Y^\circ, \mathbb{C}) 
\end{equation}

we are done. We will follow the same strategy as the book of Griffiths-Harris for the proof of the Lefschetz hyperplane theorem.

Consider 
\begin{equation}\label{eq:rest}
\Omega_X^p(\log D)\overset{r}\to \Omega_X^p(\log D)|_Y\overset{i}\to \Omega_Y^p(\log D)
\end{equation}

To compute the restriction, we consider the ideal sequence defining $Y$ twisted by $\Omega_X^p(\log D)$
\begin{equation}\label{eq:TwIdeal}
0\to \Omega_X^p(\log D)(-Y)\to \Omega_X^p(\log D)\to \Omega_X^p(\log D)_Y\to 0
\end{equation}
From the wedges of the dual SES of the normal bundle and using that $N_{Y/X}\cong \Oo_Y(Y)$ and $\Omega_Y^{p-1}(\log D)(-Y)\cong \Omega_Y^{p-1}(\log D)(N_{Y/X}^*)$ , we have that 
\begin{equation}\label{eq:WedPow}
0\to \Omega_Y^{p-1}(\log D)(-Y)\to \Omega_X^p(\log D)|_{Y}\to \Omega_Y^p(\log D)\to 0
\end{equation}

Now, as $Y$ is ample, we have the following vanishings from the logarithmic Akizuki-Nakano Theorem \cite[Corollary 6.8]{Esnault-Viehweg} 
\begin{align*}
H^q(X,\Omega_X^p(\log D)(-Y))=0 \quad p+q <n+1 \\
H^q(Y,\Omega_Y^{p-1}(\log D)(-Y))=0 \quad p+q<n+1
\end{align*}
which gives the result by using the long-exact sequences associated to the sequences above.
\end{proof}

\subsection{Connectedness} 
Consider a trivial family 
$\mathcal{X}=X\times S\to S$
and a globally trivial divisor given by $D\subset X$
$$\mathcal{D}=D\times S \to S. $$
Define $\mathcal{X}^\circ:=\mathcal{X}\setminus \mathcal{D}=D$. We have a family $\mathcal{X}^\circ\to S. $

Consider the universal family of smooth hypersurfaces intersecting $D$ transversely given by the sections of $L\to X$: $$\mathcal{Y}\to S,$$
denote the restriction $\mathcal{D}$ to $\mathcal{Y}$ by $\mathcal{D}_{\mathcal{Y}}$ and let $\mathcal{Y}^\circ:=\mathcal{Y}\setminus \mathcal{D}_{\mathcal{Y}}$.
Then the Leray spectral sequence and Zariski's theorem \ref{thm:ZarLef} imply that the map
$$H^i(\mathcal{X}^\circ, \mathbb{Q})\to H^i(\mathcal{Y}^\circ,\mathbb{Q}) $$
is an isomorphism for $i<n$ and injective for $i=n$.

\begin{thm} Let $X=\mathbb{P}^{n+1}$, $L=\Oo_{\mathbb{P}^{n+1}}(d)$, $S\subset \mathbb{P}(H^0(X,L))$ the locus of smooth hyperplane sections intersecting $D$ transversely. Suppose that $D=\sum_{i=1}^e D_i$ with $\deg D_i=d_i$. Denote $\delta_{\min}=\min\{d,d_1,\ldots, d_e\}$. If $\delta_{\min}$ is big enough, we have that 
$$H^i(\mathcal{X}^\circ, \mathbb{Q})\to H^i(\mathcal{Y}^\circ,\mathbb{Q}) $$
is an isomorphism for $i<2n$ and injective for $i=2n$.
\end{thm}
\begin{proof}
We will use the following reduction steps

\begin{prop} With the above hypothesis, the restriction map
$$H^q(S\times \mathbb{P}^{n+1}, \Omega_{S\times \mathbb{P}^{n+1}}^p(\log \mathcal{D}))\to H^q(\mathcal{Y},\Omega_{\mathcal{Y}}^p (\log \mathcal{D}_{\mathcal{Y}})) $$
is an isomorphism for $p+q<2n$, $p<n$, and it is injective for $p+q\leq 2n$, $p\leq n$.
\end{prop}

Let us denote by $p:\mathbb{P}^{n+1}\times S\to S$ and $\pi:\mathcal{Y}\to S$.
The proposition is proven by using the following
\begin{prop}\label{prop:Leray} With the above hypothesis, 
$$R^q p_*\Omega_{\mathbb{P}^{n+1}\times S}^p(\log \mathcal{D})\to R^q\pi_*\Omega_{\mathcal{Y}}^p(\log \mathcal{D}_{\mathcal{Y}}) $$
is an isomorphism for $p+q<2n$, $p<n$, and it is injective for $p+q\leq 2n$, $p\leq n$.
\end{prop}
\begin{proof}
Consider the VMHS given by $\pi^\circ:\mathcal{Y}^\circ\to S$ with local system $H_{\mathbb{C},\prim}^n=R^n(\pi^\circ)_*\mathbb{C}_{\prim}, $ where 
$$ H^*(Y_s\setminus D_s)_{\prim}:=\ker(H^*(Y_s\setminus D_s) \to H^{*+2}(X\setminus D))$$
in degree $\leq 2n$.

The Hodge bundles $\mathscr{H}^{p,q}=F^p \mathscr{H}^n/ F^{p+1}\mathscr{H^n}$ where $\mathscr{H}^n
=H_{\mathbb{C},\prim}^n\otimes \Oo_S$ have the Gauss-Manin connection $\nabla$. Consider 
$$\Gr_F^p(\DR(H_{\mathbb{C},\prim}^n)):=0\to \mathscr{H}^{p,q}\overset{\nabla}\to \mathscr{H}^{p-1,q+1}\otimes \Omega_S^1\overset{\nabla}\to \mathscr{H}^{p-2,q+2}\otimes \Omega_S^2\to \ldots $$
What makes everything work in the open case is:
\begin{thm}[Asakura-Saito] \label{thm:AS} Suppose that $D=\sum_{i=1}^e D_i$ with $\deg D_i=d_i$. Denote $\delta_{\min}=\min\{d,d_1,\ldots, d_e\}$. If $\delta_{\min}$ is big enough, the complexes $\Gr_F^p(\DR(H_{\mathbb{C},\prim}^n))$ are exact in degree $k<\min(p,n)$.
\end{thm}
\begin{proof}
This follows from Theorem 9.1 (iv) of \cite{AS06}, because in this case $c=c_S(\mathcal{X},\mathcal{D})=0$, by hypothesis $\delta_{\min}\gg 0$.
\end{proof}
We continue with the proof of Proposition \ref{prop:Leray}. Consider the Leray filtrations $L$ for $\Omega_{\mathbb{P}^{n+1}\times S}^p(\log \mathcal{D})$ and $\Omega_{\mathcal{Y}}^p(\log \mathcal{Y}_D)$:
\begin{align*}
& L^r \Omega_{\mathbb{P}^{n+1}\times S}^p(\log \mathcal{D})= p^* \Omega_{S}^r \wedge \Omega_{\mathbb{P}^{n+1}\times S}^{p-r}(\log \mathcal{D})  \\
& L^r \Omega_{\mathcal{Y}}^p(\log \mathcal{Y}_D)= \pi^* \Omega_S^r \wedge \Omega_{\mathcal{Y}}^{p-r}(\log \mathcal{Y}_D).
\end{align*}
The restriction morphism $R^q p_*\Omega_{\mathbb{P}^{n+1}\times S}^p(\log \mathcal{D})\to R^q\pi_*\Omega_{\mathcal{Y}}^p(\log \mathcal{D}_{\mathcal{Y}}) $ is compatible with the Leray filtrations. Therefore, we have to show that the restriction induces an isomorphism on the second page $E_2$ of spectral sequences, for all bidegrees $(r,u)$ with $u+r+p<2n$, $u+r<n$, and an injection for $u+r+p\leq 2n$, $u+r\leq n$.

For the first sheaf, we have that $E_1^{r,u}$ is given by 
$$\Omega_S^r\otimes H^{u+r}(\mathbb{P}^{n+1}, \Omega_{\mathbb{P}^{n+1}}^{p-r}(\log D)), $$
with differential $d_1 = \bar{\nabla}$ induces by the Gauss-Manin connection. For the second sheaf, using proposition \ref{prop:HolLerFil}, we have that $E_1^{r,u}$ can be identified with $\Omega_S^r\otimes \mathcal{H}^{p-r,u+r}$, again $d_1$ identifies with the Gauss-Manin connection.

We have a direct sum decomposition
$$H^*(Y_s\setminus D_s,\mathbb{C})=H^*(\mathbb{P}^{n+1}\setminus D,\mathbb{C})\oplus H^*(Y_s\setminus D_s,\mathbb{C})_{\prim} $$
where 
$$ H^*(Y_s\setminus D_s)_{\prim}:=\ker(H^*(Y_s\setminus D_s) \to H^{*+2}(X\setminus D))$$
in degree $\leq 2n$. We conclude that, for $u+p\leq 2n$, we have extensions 
$$0\to \Omega_S^{r}\otimes H^{u+r}(\mathbb{P}^{n+1},\Omega_{\mathbb{P}^{n+1}}^{p-r}(\log D))\to \Omega_S^r\otimes \mathcal{H}^{p-r,u+r}\to \Omega_S^r\otimes \mathscr{H}^{p-r,u+r} \to 0  $$
where every complex has its corresponding $\bar{\nabla}$ as differential.

Now, by Theorem \ref{thm:AS}, the complexes most to the right $(\Gr_F^p(\DR(H_{\mathbb{C},\prim}^n), \bar{\nabla})$ are exact in degree $k<\min(n,p)$, for $p+r=n$ and in all degree for $p+r\not= n$ it shows that, for $p+u\leq 2n$, the complex 
has cohomology $(\Omega_S^r\otimes \mathcal{H}^{p-r,u+r},\nabla)$
$$\Omega_S^r\otimes H^{u+r}(\mathbb{P}^{n+1}, \Omega_{\mathbb{P}^{n+1}}^{p-r}(\log D)), $$ in degree $r<\min(n,p)$, for $p+u=n$, and in all degree for $p+u\not = n$.

Therefore, for $p+u\leq 2n$, the restriction map induces an isomorphism for the page $E_2$ of the Leray spectral sequences of $\Omega_{S\times \mathbb{P}^{n+1}}(p)(\log \mathcal{D})$ and $\Omega_{\mathcal{Y}}^p(\log \mathcal{D}_{\mathcal{Y}})$ in degrees $(r,u)$ such that $r<\min(n,p)$ and $p+u\not = n$.
\end{proof}
\end{proof}

\subsection{Appendix: Zariski theorem by Hodge modules}

Now, we can replace the SNC condition by using Hodge modules
\begin{thm}[Zariski thm of Lefschetz type]
Let $X$ be smooth projective of dimension $n+1$. Let $D\subset X$ be a reduced hypersurface and let $Y\subset X$ be a general ample hypersurface which meets $D$ transversely. Set
 $$X^\circ=X\setminus D, \quad Y^\circ=Y\setminus (D\cap Y).$$ 
Then the restrinction map
$$Y^{\circ, *}:H^i(X^\circ, \mathbb{Q})\to H^i(Y^\circ, \mathbb{Q}) $$
induced by the inclusion $i^\circ:Y^\circ \hookrightarrow X^\circ$
is an isomorphism for $i\leq n-1$ and injective for $i=n$
\end{thm}
\begin{proof}
It suffices to work over $\mathbb{C}$. We will prove the injectivity at the level of $\Gr_F^p$.
Let $j:X^\circ \to X$ and $j':Y^\circ \to Y$ be the inclusions. Consider the polarizable Hodge modules 
$$\mathcal{M}:=j_* \mathcal{Q}_{X^\circ}^H[n+1] \text{ on }X, \quad \mathcal{M}':=j'_* \mathcal{Q}_{Y^\circ}^H[n] \text{ on }Y.$$ 

By Saito degeneration of Hodge-de Rham spectral sequence associated to $\DR(\mathcal{M})$  at $E_1$ we have that 
$$\Gr_F^p H^i(X^\circ, \mathbb{C} ) \cong \mathbb{H}^i(X, \Gr_F^p \text{DR}(\mathcal{M}))$$
and likewise for $Y^\circ$ and $\mathcal{M}'$ (see Saito \cite{Saito-MHM} and Schnell \cite{Schnell-Saito}).

Hence, if we can show that 
\begin{equation}\label{eq:injHM}
\Gr_F^p H^i(X^\circ, \mathbb{C})\hookrightarrow \Gr_F^p H^i(Y^\circ, \mathbb{C}),
\end{equation}
 we are done.

Consider the maps of Hodge modules
\begin{equation}\label{eq:restHM}
\mathcal{M}\to i_+(i^*\mathcal{M})\to i_+(\mathcal{M}')
\end{equation}
(Analogue to 
$$\Omega_X^p(\log D)\overset{r}\to \Omega_X^p(\log D)|_Y\overset{i}\to \Omega_Y^p(\log D)
 $$)

To compute the restriction, we consider the distinguished triangle in the derived category of mixed Hodge modules on $X$:
\begin{equation}\label{eq:TwIdealHM}
\mathcal{M}\otimes \Oo_X(-Y)\to \mathcal{M}\to i_+(i^*\mathcal{M})\overset{[+1]}\to .
\end{equation}

Since $Y$ is chosen general and meets $H$ transversely, the immersion $i:Y\hookrightarrow X$ is non-characteristic for the filtered $\mathcal{D}$-module underlying $\mathcal{M}$. Therefore, the filtered inverse image is well-behaved and one has a codimension-one relation for the de Rham complex  (see Popa's exposition \cite{Popa-notes}):
$$Li^*\DR_X(\mathcal{M})\cong \DR_Y(i^*\mathcal{M})[-1]. $$
This permits us to pass to $Y$. Taking $\Gr_F$ we obtain the exact triangle 

\begin{equation}\label{eq:WedPowMHM}
\Gr_{F}^{p-1}\text{DR}_Y(\mathcal{M}'(-Y))[-2] \to \Gr_F^p(\DR_Y(i^*\mathcal{M})[-1]) \to \Gr_F^p \text{DR}(\mathcal{M}')\overset{[+1]} \to  
\end{equation}

(Analogue to
$$0\to \Omega_Y^{p-1}(\log D)(-Y)\to \Omega_X^p(\log D)|_{Y}\to \Omega_Y^p(\log D)\to 0$$)

Now, by Saito vanishing, we have that as $\Oo_X(Y)$ is ample
$$\mathbb{H}^k(X, \Gr_F^p \text{DR}(\mathcal{M}) \otimes \mathcal{O}_X(-Y))=0 \quad \text{for } k< 0$$
$$\mathbb{H}^k(Y, \Gr_F^p \text{DR}(\mathcal{M}' \otimes \mathcal{O}_Y(-Y))) = 0 \quad \text{for } k < 0$$

which gives the result by using the long-exact sequences associated to the above distinguished triangles.

 \end{proof}

\section{Hodge Loci}\label{S:RelNL}
Let $Y\subset \mathbb{P}^n$ be a smooth hypersurface of degree $e$. Let $S\subset H^0(\Oo_{\mathbb{P}^n}(d))$ be the locus of smooth hypersurfaces intersecting transversally $Y$. For every $s\in S$ denote $Z_s:=X_s\cap Y$.

Consider the universal family $\mathcal{X}\to S$ and denote $\mathcal{X}^\circ:=\mathcal{X}\setminus Y\times S$.

Consider the VMHS given by $\pi^\circ:\mathcal{X}^\circ\to S$. We denote by $\mathcal{H}_{\text{prim}}^{n-1}$ the local system of primitive cohomology:$$\mathcal{H}_{\text{prim}}^{n-1} \subset R^{n-1}(\pi^\circ)_*\mathbb{C}.$$
Here, $$H^{n-1}(U)_{\text{prim}} := \text{Coker}(H^{n-1}(\mathbb{P}^n \setminus Y) \to H^{n-1}(X \setminus Y)).$$ 
We define the Hodge loci $S_\lambda^p$ by 
$$S_\lambda^p:=\{s\in S\mid \lambda_s\in F^p \mathcal{H}_{\text{prim},s}\}. $$

\subsection{Generalized Jacobian Rings}
Here we give a brief overview of the main results of \cite{AS06} that we will need.

Let $X=\{F=0\}, Y=\{G=0\}\subset \mathbb{P}^n$ with $\deg F=d, \deg G=e$. Denote by $\delta_{\min}=\min\{d,e\}$. Let $P=\mathbb{C}[X_0,\ldots, X_n]$ and $A=P [\mu, \nu]$. Let $P^l\subset P$ denote the subspace of homogeneous polynomials of degree $l$.

For $q, l\in \mathbb{Z}$, write:

$$ A_q(l)=\oplus_{a+b=q}P^{ad+be+l}\mu^a\nu^b, \quad A_q(l)=0 \text{ if } q<0. $$

The Jacobian ideal $J(F,G)$ is the ideal of $A$ generated by 
$$ \frac{\partial F}{\partial X_k}\mu +\frac{\partial G}{\partial X_k}\nu \quad 0\leq k\leq n $$

The quotient ring $B=B(F,G)=B(X,Z)=A/J(F,G)$ is called the Jacobian ring of the pair $(X,Z)$. We denote
$$B_q(l)=B_q(l)(F,G)=A_q(l)/A_q(l)\cap J(F,G). $$

we always assume that $X$ and $Y$ intersect transversely.

\begin{thm}[Isomorphism Theorem]\label{thm:AS-I}
Suppose $n\geq 2$. 
\begin{enumerate}
    \item For integers $0 \leq q \leq n-1$ and $l \geq 0$, there is a natural isomorphism
    $$ \varphi: B_q(d+e-n-1+l) \xrightarrow{\sim} H^q(X, \Omega_{X}^{n-1-q}(\log Z)(l))_{\prim}. $$
    
    \item Under this isomorphism, the infinitesimal variation of Hodge structure (Gauss-Manin connection)
    $$ \bar{\nabla}: H^q(X, \Omega_{X}^{n-1-q}(\log Z)(l))_{\prim} \otimes T_S \longrightarrow H^{q+1}(X, \Omega_{X}^{n-q-2}(\log Z)(l))_{\prim} $$
    is identified with the multiplication map in the Jacobian ring:
    $$ B_q(d+e-n-1+l) \otimes B_1(0) \longrightarrow B_{q+1}(d+e-n-1+l). $$
\end{enumerate}
\end{thm}

\begin{thm}[Duality Theorem]\label{thm:AS-II}Suppose $n\geq 2$. Let $\Sigma = 2(d-n-1)+e$. The natural pairing $$B_q(l) \times B_{n-1-q}(\Sigma - l) \longrightarrow B_{n-1}(\Sigma) \cong \mathbb{C}$$is perfect for any $l$ such that $d-n-1\le l\le d-n-1+e$ and any $0 \leq q \leq n-1$.
\end{thm}

\begin{rem}In \cite{AS06}, the pairing is defined more generally and fails to be perfect in certain ranges due to a defect space. However, under our conditions for $X$ and $Y$ (both smooth hypersurfaces), the defect vanishes and the duality is perfect. This is the only case we require.\end{rem} 

\subsection{Hypersurface case and smooth divisor}
We use the above notation

\begin{cor}\label{cor:HodgeFil} 
Suppose that $d-n-1\le l\le d-n-1+e$. Let $\Sigma = 2(d-n-1)+e$. We have that $B_p(l) \neq 0$ if and only if:
\begin{itemize}
    \item     $$ \delta_{\min} \cdot p + l \ge 0 $$
    \item     $$ l \leq \Sigma + \delta_{\min} (n-1-p). $$
\end{itemize}
\end{cor}

\begin{proof}
The conditions $d-n-1\le l\le d-n-1+e$ and $2 \leq n$ imply that we are in the situation of Theorem \ref{thm:AS-II}. This establishes a perfect duality pairing:
$$ B_p(l) \times B_{n-1-p}(\Sigma - l) \longrightarrow B_{n-1}(\Sigma) \cong \mathbb{C}. $$
Consequently, $B_p(l) \neq 0$ if and only if both the space itself and its dual space are non-zero.

For the Jacobian ring component $B_p(l)$ to be non-zero, the underlying polynomial space $A_p(l)$ must be non-zero. The minimal degree of a polynomial in $A_p(l)$ is $\delta_{\min} \cdot p + l$. 
If $\delta_{\min} \cdot p + l < 0$, then $A_p(l) = 0$, hence $B_p(l) = 0$. 
If $\delta_{\min} \cdot p + l \ge 0$, then $A_p(l) \neq 0$. 

 By perfect duality, $B_p(l) \neq 0 \iff B_{n-1-p}(\Sigma - l) \neq 0$. We apply the same non-negativity condition to the dual space indices:
\begin{itemize}
    \item Exterior degree: $p' = n-1-p$.
    \item Polynomial twist: $l' = \Sigma - l$.
\end{itemize}
The condition for the dual space to be non-zero is:
$$ \delta_{\min}(n-1-p) + (\Sigma - l) \ge 0. $$
Rearranging this inequality gives $l \leq \Sigma + \delta_{\min}(n-1-p)$.

Therefore, $B_p(l)$ is non-zero if and only if $l$ satisfies the interval defined by the non-negative degree of the space itself (lower bound) and the non-negative degree of its dual (upper bound).
\end{proof}

\begin{rem}
In the case of general type hypersurfaces where \(d \ge n+1\), the perfect duality window \(d-n-1 \le l \le d-n-1+e\) is contained within the non-vanishing range of Corollary \ref{cor:HodgeFil}. However, for Fano or Calabi-Yau hypersurfaces (\(d \le n\)), one must strictly enforce the non-vanishing inequalities \(\delta_{\min}p+l \ge 0\) and \(l \le \Sigma + \delta_{\min}(n-1-p)\), as the duality window may partially extend into the vanishing range where the spaces are trivially zero.
\end{rem}

\begin{cor}\label{cor:ConsMac} 
Let $\Sigma = 2(d-n-1)+e$. Consider integers $p, p', l, l'$ such that
\begin{itemize}
    \item $\delta_{\min}(n-1-p-p') + \Sigma - (l+l') \geq 0$, 
    \item $\delta_{\min} \cdot p' + l' \geq 0$,
    \item $p+p'\leq n-1$,
    \item $l+l' \leq \Sigma$, and 
    \item $d-n-1\le l\le d-n-1+e$.
\end{itemize} 
Then the map
$$ \mu: B_{p}(l) \to \Hom(B_{p'}(l'), B_{p+p'}(l+l')) $$
is injective.
\end{cor}

\begin{proof}
By contradiction. Suppose we have an element $A \in \ker \mu \subset B_{p}(l)$. By definition, for all $B \in B_{p'}(l')$, we have that 
$$ AB = 0 \in B_{p+p'}(l+l'). $$
We wish to show $A=0$ by testing it against the dual space. For all $\phi \in B_{n-1-p-p'}(\Sigma - l - l')$, we have 
$$ A(B\phi) = (AB)\phi = 0 \in B_{n-1}(\Sigma). $$
The first two bullet points of the hypothesis ensure that the degrees of the generating spaces are non-negative, so the spaces are non-zero. Since the Jacobian ring $B$ is a quotient of the polynomial ring $A$ (where multiplication is surjective), and the graded pieces are non-zero, the multiplication map
$$ B_{n-1-p-p'}(\Sigma - l - l') \otimes B_{p'}(l') \to B_{n-1-p}(\Sigma - l) $$
is surjective. 
Therefore, the products $C = B\phi$ generate the entire space $B_{n-1-p}(\Sigma - l)$, which is the dual of $B_p(l)$.
We have established that $AC = 0$ in $B_{n-1}(\Sigma)$ for all $C$ in the dual space of $B_{p}(l)$. By the fifth bullet point, Theorem \ref{thm:AS-II} applies and the duality is perfect; thus we obtain that $A=0$.
\end{proof}

\begin{thm} 
Let $\Sigma = 2(d-n-1)+e$. If
\begin{itemize}
    \item $\delta_{\min}(n-p) + d+e-n-1 \ge 0$, and
    \item $d+e-n-1 \leq \Sigma + \delta_{\min}(p-1)$,
\end{itemize}
the Hodge loci $S_\lambda^p$ are proper analytic subspaces of $S$, for all $0\not = \lambda \in H^0(S, R^{n-1}\pi_*\mathbb{C}_{\mathrm{prim}})$.
\end{thm}

\begin{proof}
The first two conditions and Corollary \ref{cor:HodgeFil} (applied to $q=n-p$) imply that 
$$
H^{p-1, n-p}(X_s \setminus Z_s) \cong B_{n-p}(d+e-n-1) \neq 0.
$$
Since $F^p \subset F^{p-1}$, showing $S_\lambda^p$ is proper is equivalent to showing that $\lambda$ does not remain in $F^p$ everywhere, i.e., that the derivative of $\lambda$ moves it into a non-zero part of $F^{p-1}/F^p$. Specifically, we want to show that if $\nabla \lambda = 0$ in the graded piece, then $\lambda = 0$.

Let $s\in S$ and suppose that there exists $0\neq \lambda\in H^{p,n-p-1}(X_s\setminus Z_s)$. Let us take a polynomial $P\in B_{n-p-1}(d+e-n-1)$ representing $\lambda$.

The infinitesimal variation $\bar{\nabla}(\lambda)$ identifies with the multiplication map by $P$ by Theorem \ref{thm:AS-I}:
$$
\mu_P: B_1(0) \to B_{n-p}(d+e-n-1). 
$$
We claim that this map $\mu_P$ is injective (as a map of vector spaces, meaning $P$ acts faithfully). To see this, consider the multiplication pairing map:
$$
\mu: B_{n-p-1}(d+e-n-1) \to \Hom(B_1(0), B_{n-p}(d+e-n-1)).
$$
By Corollary \ref{cor:ConsMac}, this map $\mu$ is injective. Therefore, $\mu(P) = \mu_P$ is zero if and only if $P=0$. 
Consequently, the infinitesimal period map
$$
\nabla: T_{S,s} \to H^{p-1,n-p}(X_s\setminus Z_s) 
$$
is non-zero for $0\neq \lambda$.

Standard arguments involving Griffiths transversality show that if $S_\lambda^p = S$ (meaning $\lambda$ is a flat section lying in $F^p$ everywhere), then $\nabla \lambda = 0$. As shown above, this implies $P=0$, so the component of $\lambda$ in $H^{p, n-p-1}$ vanishes, meaning $\lambda \in F^{p+1}$.
By induction on $p$, we conclude that if $S_\lambda^p=S$, we have $\lambda=0$.
\end{proof}

\begin{rem}
\label{rem:ell-range-check}
For the representative $P\in B_{n-p-1}(d+e-n-1)$, the twist is $l=d+e-n-1$. This satisfies the perfect duality window condition $d-n-1 \le l \le d-n-1+e$ trivially because $e\geq 0$, ensuring Theorem \ref{thm:AS-II} applies.

\end{rem}

\bibliographystyle{smfalpha}
\bibliography{sample}
\end{document}